\documentclass{ims9x6}
\usepackage{amsmath,mathtools}
\makeindex



\begin{document}

\setcounter{page}{13}  

\chapter{PATRICIA BRIDGES}

\markboth{Steve Evans \& Anton Wakolbinger}{PATRICIA bridges} 

%

\author{Steven N. Evans}
\address{Department of Statistics\\
         University  of California\\ 
         Berkeley, CA 94720-3860, U.S.A. \\
         evans@stat.berkeley.edu}

\author{Anton Wakolbinger}
\address{Institut f\"ur Mathematik \\
         Goethe-Universit\"at \\
         60054 Frankfurt am Main, Germany\\
         wakolbinger@math.uni-frankfurt.de}
         



\begin{abstract} A {\em radix sort tree} arises when storing distinct infinite binary words in the leaves of a binary tree such that for any two words their common prefixes  coincide with the common prefixes of the  corresponding two leaves. If one deletes the out-degree $1$ vertices in the radix sort tree and ``closes up the gaps'', then the resulting {\em PATRICIA tree} maintains all the information that is necessary for sorting
the infinite words into lexicographic order. 
 
We investigate the  PATRICIA chains
-- the tree-valued Markov chains that arise when successively building the PATRICIA trees for the collection of infinite binary words 
$Z_1,\ldots, Z_n$, $n=1,2,\ldots$, 
where the source words $Z_1, Z_2,\ldots$ are independent and have a common diffuse distribution on $\{0,1\}^\infty$. 
It turns out that the PATRICIA chains share a common collection of backward transition probabilities and that these are the same
as those of a chain introduced by R\'emy for successively generating uniform random binary trees with larger and larger numbers of leaves.
This means that the infinite bridges of any PATRICIA chain (that is, the chains obtained by conditioning a PATRICIA chain on its remote future)
coincide with the infinite bridges of the R\'emy chain.
The infinite bridges of the R\'emy chain are characterized concretely in Evans, Gr\"ubel, and Wakolbinger 2017
and we recall that characterization here while adding some details and clarifications.\end{abstract}

\section{Introduction}
\label{S:intro}
We first fix some notation. Denote by $\{0,1\}^\star:=\bigsqcup_{k=0}^\infty\{0,1\}^k$ the set of finite tuples
or {\em words} drawn from the alphabet $\{0,1\}$ (with the empty word $\emptyset$
allowed) -- the symbol $\bigsqcup$ emphasizes that this is a disjoint union.  
Write an $\ell$-tuple $v=(v_1, \ldots, v_\ell) \in \{0,1\}^\star$ more simply as
$v_1 \ldots v_\ell$ and set $|v| = \ell$. 

Define a partial order on $\{0,1\}^\star \cup \{0,1\}^\infty$ by declaring that
$u < v$ if and only if $u$ is a strict prefix of $v$; that is, $u < v$ if and only if
$u = u_1 \ldots u_k$ for some $k \ge 0$, $v$ is of the form $v_1 \ldots v_\ell \in \{0,1\}^\star$
where $k < \ell$ or $v_1 v_2 \ldots \in \{0,1\}^\infty$, and $v_1 \ldots v_k = u_1 \ldots u_k$.
The empty word is the unique minimal element for this partial order.

A {\em  rooted binary tree} (or {\em binary tree} for short)
is a non-empty  subset $\mathbf{t}$ of 
$\{0,1\}^\star$ with the property that if $v \in \mathbf{t}$
and $u \in \{0,1\}^\star$ is such that $u < v$, then $u \in \mathbf{t}$.
The vertex $\emptyset$ (that is, the empty word)
belongs to any such tree $\mathbf{t}$ and
is the {\em root} of $\mathbf{t}$. For a finite rooted binary tree $\mathbf{t}$ the {\em leaves} of $\mathbf{t}$ are the maximal elements of $\mathbf{t}$
for the partial order $<$.
A finite rooted binary tree is uniquely determined by its
leaves: it is the smallest rooted binary tree that contains the set of leaves and it
consists of the leaves and the elements $u \in \{0,1\}^\star$ such that 
$u < v$ for some leaf $v$. 

Suppose now that $z_1, \ldots, z_n \in \{0,1\}^\infty$ are distinct infinite binary words.  For each $i \in [n]$ we may construct a finite binary word $y_i$ that is an initial segment of $z_i$ such that
$y_1, \ldots, y_n$ are the distinct leaves of a binary tree and $y_1, \ldots, y_n$ are the minimal length words with this property.  The resulting binary tree is called the {\em radix sort tree} defined by the infinite words
$z_1, \ldots, z_n$: a depth first search of this tree visits the leaves in an order that coincides with the lexicographic order of the corresponding infinite words.
The radix sort tree stores more information than is necessary for the purpose of sorting the infinite binary words into lexicographic order. 
More precisely, if one deletes the out-degree $1$ vertices in the radix sort tree and ``closes up the gaps'', then a depth first search of the resulting
{\em PATRICIA tree} still visits the leaves in an order that coincides with the lexicographic order of the corresponding infinite words.  
PATRICIA is
an acronym for ``Practical Algorithm To Retrieve Information Coded In Alphanumeric''.
PATRICIA trees were invented independently in \cite{Mor68} \cite {Gwe68}. 
Note that a PATRICIA tree is a 
{\em full binary tree}: each non-leaf vertex of the tree has two children; that is, if the finite binary word $v = v_1 \ldots v_m$ is a vertex of the tree that is not a leaf,
then both of the words $v_1 \ldots v_m 0$ and $v_1 \ldots v_m 1$ are also vertices of the tree.

Suppose now that $Z_1, Z_2, \ldots$ is an infinite i.i.d. sequence of random elements of $\{0,1\}^\infty$ with common diffuse distribution $\nu$.  Let $\prescript{\nu}{}{\bar R}_n$, $n \in \mathbb{N}$, be
the PATRICIA tree constructed from $Z_1, \ldots, Z_n$.  We show in Proposition~\ref{P:Markov_property} that $(\prescript{\nu}{}{\bar R}_n)_{n \in \mathbb{N}}$ is a Markov chain which we call a PATRICIA chain. 
Features
of PATRICIA trees for random inputs were first studied in \cite{MR0445948}
and this topic has since been the subject of quite a large literature (see, for example,
\cite{Leckey_thesis}, \cite{MR3741547}, \cite{MR3429094}, \cite{MR3517923}, \cite{MR3248350}, \cite{MR3247579}, \cite{MR3192399},  \cite{MR2131826},   \cite{MR1932678}, \cite{MR1910516}, \cite{MR1871557}, \cite{MR1215228}, \cite{MR1192779},  \cite{MR1151362}, \cite{MR1083648}, \cite{MR1022654}).

Our aim in this paper is to characterize the
infinite bridges of the PATRICIA chains; that is, for each $\nu$ we wish to characterize the family of Markov chains that have the same backward transition probabilities as 
$(\prescript{\nu}{}{\bar R}_n)_{n \in \mathbb{N}}$.  
As we observe in Proposition~\ref{P:Markov_property}, the backward transition probabilities of $(\prescript{\nu}{}{\bar R}_n)_{n \in \mathbb{N}}$ are the same for all $\nu$ and this common
backward transition mechanism may be described as follows: pick a leaf uniformly at random, delete it and its sibling, and close up the gap in the tree if there is one.
It follows that all the PATRICIA chains have the same infinite bridges.  

Before we proceed, we need to recall the R\'emy chain of \cite{MR803997} (see, also, \cite{MR1331596}, \cite{MR2042386}, \cite{MR1714707}).
The state space of this chain is also the set of full binary trees.
Writing $\aleph$ for the full binary tree with two leaves, the R\'emy chain evolves forwards in time as follows:
\begin{itemize}
\item
Pick a
vertex $v$ uniformly at random.
\item
Cut off the subtree rooted at $v$ and set it aside.
\item
Attach a copy of the tree $\aleph$ to the end of the edge
that previously led to $v$.
\item
Re-attach the subtree that was rooted at $v$ in  uniformly at random
to one of the two leaves in the copy of $\aleph$.
\end{itemize}
The R\'emy chain started with the trivial tree consisting of a single vertex has the important feature that after $n$ steps it
is uniformly distributed over the set of full binary trees with $n$ leaves.
The backward transition probabilities of the R\'emy chain were identified in \cite{MR3601650}
and they coincide with those of the PATRICIA chains.  It follows that
the common family of infinite bridges for the PATRICIA chains is the same as that of the R\'emy chain and the latter was determined in \cite{MR3601650}.

An outline of the remainder of the paper is as follows.

In Section~\ref{S:trees} we introduce some notation and recall more formally the construction of a radix sort tree and
a PATRICIA tree from a set of infinite binary words.  In
Section~\ref{S:chains} we begin considering the stochastic processes built by applying the radix sort tree and PATRICIA
tree constructions to sequences of independent identically distributed infinite binary words with common diffuse 
distribution $\nu$.  We observe that these processes are Markov, show that all the radix sort chains have the same
family of backward transition probabilities, and demonstrate 
that the same is true of the PATRICIA chains. We also investigate the
issue that different diffuse probabilities $\nu$ can give rise to the same PATRICIA chain.

Section~\ref{S:bridges} initiates the study of infinite bridges {\em per se}.  
Since both in the case of the radix sort chains and of the PATRICIA chains the backward transition probabilities are the same for all $\nu$, we concentrate on the case where
$\nu$ is fair coin-tossing measure.  We also recall that the distribution of any infinite bridge for
either the radix sort chain or the PATRICIA chain will be a mixture of extremal infinite bridge distributions
and that an infinite bridge is extremal if and only if it has an almost surely trivial tail $\sigma$-field, so it suffices 
for the determination of the distributions of general infinite bridges to characterize the extremal ones.
We recall from \cite{MR3734107} that the extremal infinite bridges for the radix sort chain are nothing other than
the radix sort chains with general diffuse input distribution $\nu$.  By the observation that all the PATRICIA chains have
the same backward transition probabilities, it follows that a PATRICIA chain with general diffuse input distribution $\nu$ is
an infinite bridge for our ``reference'' PATRICIA chain with fair coin-tossing input distribution.  Because of the
characterization of the extremal infinite bridges for the radix sort chain that we have just recalled and
the fact that if $\Phi$ is the ``delete the out-degree $1$ vertices'' map that turns a 
binary tree into a full binary tree and $(R_n^\infty)_{n \in \mathbb{N}}$ is an infinite bridge
for the radix sort chain, then $(\Phi(R_n^\infty))_{n \in \mathbb{N}}$ is an infinite bridge
for the PATRICIA chain (see Proposition~\ref{P:bridges_to_bridges}), one might expect that such infinite PATRICIA bridges 
exhaust the collection of
extremal infinite PATRICIA bridges.  We present an example to show that this is not the case.  Rather, to describe
the totality of the extremal infinite PATRICIA bridges we need to recall from \cite{MR3601650} the description of
the extremal infinite bridges for the R\'emy chain because, as we have already noted, these two families coincide.

This retelling takes up the remainder of the paper with the exception of
the brief interpolation Section~\ref{S:heights} where we present some examples
of infinite PATRICIA bridges and give an indication of how differently such
chains can behave by examining the asymptotic growth of the successive trees
they produce.

A crucial element of the development in \cite{MR3601650} was the introduction of the notion of a {\em didendritic system}
with a given label set.
This class of objects is a generalization of the class of finite leaf-labeled full binary trees that includes objects with 
infinitely many ``leaves''.
Section~\ref{S:DDS} contains a review of some facts about didendritic systems.  In particular, we show that 
the class of didendritic systems with finite label sets is in a natural bijective correspondence 
with the class of finite leaf-labeled full binary trees.  
This fact was claimed in \cite{MR3601650} with a slightly different definition of didendritic system
but we show by an example that under the axioms in \cite{MR3601650} this assertion is false.  
We stress that this
gap does not illegitimate the further development in \cite{MR3601650} 
which does not depend on the details of the definition
of a didendritic system but only on the two facts that a finite didendritic system is effectively a finite leaf-labeled
binary tree and that the class of didentritic systems is closed under projective limits in a natural way -- both of which
are true for our new definition.  We also 
develop in an explicit manner an alternative description of the class of finite didendritic systems that was
somewhat implicit in \cite{MR3601650}.  This alternative description, which is obtained in Proposition~\ref{P:left_right}, is crucial for the later exposition.

We recall in Section~\ref{S:bridges_DDS} that there is a bijective correspondence between infinite PATRICIA (equivalently,
R\'emy) bridges and random didentric systems with label set $\mathbb{N}$ that are exchangeable in a natural sense.  Moreover, 
the extremal infinite PATRICIA bridges correspond to the exchangeable random didendritic systems with label set $\mathbb{N}$ that
are ergodic, where ergodicity is equivalent to the property that the random didendritic systems induced on disjoint subsets
of $\mathbb{N}$ are independent.  In order to characterize the class consisting of all PATRICIA bridges it therefore suffices
to determine the class of ergodic exchangeable random didendritic systems.  We recall the concrete representation
 of the latter from \cite{MR3601650} in Section~\ref{S:DDS_from_real_tree}.  The ingredients of this representation are
a rooted complete separable $\mathbb{R}$-tree, a diffuse probability measure on the $\mathbb{R}$-tree, and a possibly random mechanism 
for giving a left-versus-right ordering to the tree built by sampling countably many points from the $\mathbb{R}$-tree according to the given probability measure.

Finally, we note that an alternative concrete representation of extremal infinite R\'emy bridges has recently been given in \cite{Ger18} as part of a program that extends the study of Markov chains with  R\'emy-like transition probabilities to classes of discrete structures other than binary trees; for example, \cite{Ger18}
also considers the infinite bridges investigated in \cite{MR3851828} that
are similar to R\'emy or PATRICIA infinite bridges but vertices of the successive trees may have
more than two offspring and there is no left-to-right ordering of
offspring.

\section{Radix sort trees and PATRICIA trees}
\label{S:trees}
For a finite rooted binary tree $\mathbf{t}$, we write  $\mathbf{L}(\mathbf{t})$ for the set of leaves of $\mathbf{t}$.
For $y_1, \ldots, y_m \in \{0,1\}^\star$, write
\[
\mathbf T(y_1, \ldots, y_m) := \bigcup_{j=1}^m \{u \in \{0,1\}^\star : u \le y_j\}
\]
for the smallest finite rooted binary tree
containing $y_1, \ldots, y_m \in \{0,1\}^\star$; the leaves of this tree form a subset
of $\{y_1, \ldots, y_m\}$ and this subset is proper if and only if $y_i < y_j$ for some
pair $1 \le i \ne j \le m$.

A collection $z_1, \ldots, z_n$ of distinct elements of $\{0,1\}^\infty$ determines
a finite rooted binary tree in the following manner. For $1 \le i \le n$ let
$y_i$ be the minimal length prefix of $z_i$ that differs from all prefixes of
$z_j$, $j \ne i$.  The words $y_1, \ldots, y_n$ are distinct and, moreover, they
are incomparable for the partial order $<$.  
The {\em radix sort tree determined by the input} $z_1, \ldots, z_n$ is the
finite rooted binary tree $\mathbf R(z_1, \ldots, z_n)$ with leaves $y_1, \ldots, y_n$; that is,
\[
\mathbf R(z_1, \ldots, z_n) := \mathbf T(y_1, \ldots, y_n).
\]

\begin{remark}
\label{R:permutation_invariance}
Observe that  
\[
\mathbf R(z_1, \ldots, z_n) = \mathbf R(z_{\sigma(1)}, \ldots, z_{\sigma(n)})
\]
for any permutation $\sigma$ of $[n]$.
\end{remark}

\begin{notation}
Denote by $\mathbb{S}$ the class of binary trees that can arise as radix sort trees.
A binary tree $\mathbf{s}$ belongs to $\mathbb{S}$ if and only if $\mathbf{s} = \{\emptyset\}$ or $\mathbf{s}$ has at least two leaves
and for any leaf $u_1 \ldots u_p \in \mathbf{s}$ the word $u_1 \ldots u_{p-1} \bar u_p$ is also
a vertex of $\mathbf{s}$, where we set $\bar 0 := 1$ and $\bar 1 := 0$.  
Write $\mathbb{S}_n$, $n \in \mathbb{N}$, for the elements of $\mathbb{S}$ with $n$ leaves
(in particular, $\mathbb{S}_1$ contains only the trivial tree
with the single vertex $\{\emptyset\}$).
\end{notation}

One of the reasons for constructing the radix sort tree for a set of inputs
$z_1, \ldots, z_n$ is that if $y_1, \ldots, y_n$ are the leaves of the tree
(where the indexing is such that $y_i < z_i$ for $1 \le i \le n$), then
the lexicographic ordering of $z_1, \ldots, z_n$ is the same as that of $y_1, \ldots, y_n$;
that is, 
$\sigma$ is the unique permutation of $[n]$ such that $z_{\sigma(i)}$, $i \in [n]$,
are increasing in the lexicographic order 
if and only if
$\sigma$ is the unique permutation of $[n]$ such that $y_{\sigma(i)}$, $i \in [n]$,
are increasing in the lexicographic order.  We now describe another procedure for associating
a finite set of inputs with the leaves of a finite binary tree that shares this feature but
typically uses a finite binary tree with fewer vertices. 

\begin{notation}
Denote by $\bar {\mathbb{S}}$ the set of finite binary tree for which all vertices have out-degree $2$ or $0$.
That is, $\bar {\mathbb{S}}$ consists of finite binary trees $\mathbf{t}$ such that if $v \in \mathbf{t}$, then either
$v \in \mathbf{L}(\mathbf{t})$ or both $v0 \in \mathbf{t}$ and $v1 \in \mathbf{t}$.  We call the elements of $\bar {\mathbb{S}}$
{\em full finite binary trees}.
Write $\bar {\mathbb{S}}_n$, $n \in \mathbb{N}$, for the elements of $\bar {\mathbb{S}}$ with $n$ leaves
(in particular, $\bar {\mathbb{S}}_1$ contains only the trivial tree
with the single vertex $\{\emptyset\}$).
Note that $\#\mathbb{S}_n = C_{n-1}$, where $C_m := \frac{1}{m+1} \binom{2m}{m}$ is
the $m^{\mathrm{th}}$ {\em Catalan number} (see, for example, Bijective Exercise 7 of \cite{MR3467982}).
\end{notation}

\begin{definition}
The {\em PATRICIA contraction} is a map $\Phi: \mathbb{S} \to \bar {\mathbb{S}}$.  
It maps $\mathbb{S}_n$ onto $\bar {\mathbb{S}}_n$, $n \in \mathbb{N}$.  It is defined
recursively as follows.  For $n=1$, put $\Phi(\{\emptyset\}) := \emptyset$.
Now consider $\mathbf{s} \in \mathbb{S}_n$ for $n\ge 2$.  There is a unique
maximal $m$ and $u_1 \ldots u_m \in \{0,1\}^m$ such that 
$u_1 \ldots u_m < y$ for every leaf $y \in \mathbf{L}(\mathbf{s})$.  We have a decomposition
\[
\mathbf{s} = \{ \{\emptyset\}, u_1, u_1 u_2, \ldots, u_1 u_2 \ldots u_m\}
\sqcup u_1 \ldots u_m 0 \mathbf{s}^{(0)}
\sqcup u_1 \ldots u_m 1 \mathbf{s}^{(1)}
\]
for two binary trees $\mathbf{s}^{(0)}, \mathbf{s}^{(1)} \in \mathbb{S}$ that both have fewer leaves than $\mathbf{s}$.
Put
\[
\Phi(\mathbf{s}) := \{\emptyset\} \sqcup 0 \Phi(\mathbf{s}^{(0)}) \sqcup 1 \Phi(\mathbf{s}^{(1)}).
\]
\end{definition}
\begin{figure}
\centering
\includegraphics[height=3cm]{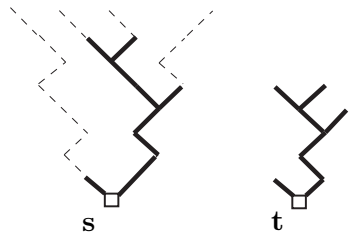}
\caption
{Initial segments of 4 elements $z_1,z_2,z_3,z_4 \in \{0,1\}^\infty $, the radix sort tree $\mathbf s = \mathbf R(z_1,z_2,z_3,z_4)$ determined by $z_1,z_2,z_3,z_4$  and its PATRICIA contraction $\mathbf t = \mathbf \Phi(\mathbf s)$.}
\label{radixandPAT}
\end{figure}
\begin{remark}
The PATRICIA contraction of $\mathbf{s} \in \mathbb{S}$ is the unique  $\mathbf{t} \in \bar {\mathbb{S}}$
with the following properties.
\begin{itemize}
\item
Writing $\mathbf{S}(\mathbf{s})$ (respectively, $\mathbf{S}(\mathbf{t})$) for the vertices of $\mathbf{s}$ (respectively, $\mathbf{t}$) with
out-degree $2$,
there is a bijective correspondence between $\mathbf{S}(\mathbf{s})$ and $\mathbf{S}(\mathbf{t})$ and 
and a bijective correspondence between
$\mathbf{L}(\mathbf{s})$ and $\mathbf{L}(\mathbf{t})$.
\item
Suppose that $w,x \in \mathbf{s}$ and $y,z \in \mathbf{t}$ are such that $w$ corresponds to $y$
and $x$ corresponds to $z$.  Then $w = u_1 \ldots u_p$ and 
$x = u_1 \ldots u_p u_{p+1} \ldots u_q$ for some $u_1, \ldots, u_q \in \{0,1\}$
with $u_{p+1} =0$ (resp. $u_{p+1}=1$) if and only if $y = v_1 \ldots v_m$ and 
$z = v_1 \ldots v_m v_{m+1} \ldots v_n$ for some $v_1, \ldots, v_n \in \{0,1\}$
with $v_{m+1} =0$ (resp. $v_{m+1}=1$).
\end{itemize}
\end{remark}

\section{Radix sort chains and PATRICIA chains}
\label{S:chains}
\begin{definition}
Given a diffuse probability measure $\nu$ on $\{0,1\}^\infty$, let $(Z_n)_{n \in \mathbb{N}}$
be a sequence of independent identically distributed $\{0,1\}^\infty$-valued random
variables with common distribution $\nu$.  The corresponding {\em radix sort chain}
$(\prescript{\nu}{}{R}_n)_{n \in \mathbb{N}}$ is defined by 
$\prescript{\nu}{}{R}_n := \mathbf R(Z_1, \ldots, Z_n)$, $n \in \mathbb{N}$,
and the corresponding {\em PATRICIA chain} $(\prescript{\nu}{}{\bar R}_n)_{n \in \mathbb{N}}$
is defined by $\prescript{\nu}{}{\bar R}_n = \Phi(\prescript{\nu}{}{R}_n)$, $n \in \mathbb{N}$.
Setting $\gamma := \pi^\infty$, where
$\pi(\{0\}) = \pi(\{1\}) = \frac{1}{2}$, we will write
$R_n := \prescript{\gamma}{}{R}_n$
and
$\bar R_n := \prescript{\gamma}{}{\bar R}_n$.
\end{definition}

The processes $(\prescript{\nu}{}{R}_n)_{n \in \mathbb{N}}$ and $(\prescript{\nu}{}{\bar R}_n)_{n \in \mathbb{N}}$
are indeed Markovian.  This observation does not seem to have appeared explicitly in the literature except
in \cite{MR3734107} for the former process.  We will state this formally along with a description of
the respective backward transition probabilities, but we first need some notation.  

\begin{notation}
Consider $\mathbf{t} \in \mathbb{S}_{n+1}$ and let $v = v_1 \ldots v_m$ be a leaf of $\mathbf{t}$.  

Suppose first that $v_1 \ldots v_{m-1} \bar v_m$ 
is not a leaf of $\mathbf{t}$. Let $\kappa(\mathbf{t},v) \in \mathbb{S}_n$ be the tree $\mathbf{t} \setminus \{v\}$.
That is, $\kappa(\mathbf{t}, v)$ is the tree with the same leaf set as  $\mathbf{t}$ except that $v$ has been removed.

On the other hand, suppose that $v_1 \ldots v_{m-1} \bar v_m$ is also a leaf of $\mathbf{t}$.
There is a largest $\ell < m$ such that
$v_1 \ldots v_\ell$ and $v_1 \ldots v_{\ell-1} \bar v_\ell$ are both  vertices of $\mathbf{t}$.
In this case let $\kappa(\mathbf{t},v) \in \mathbb{S}_n$ be the tree 
$\mathbf{t} \setminus (\{v_1 \ldots v_p : \ell+1 \le p \le m\} \cup \{v_1 \ldots v_{m-1} \bar v_m\})$
That is, $\kappa(\mathbf{t}, v)$ is the tree with the same leaf set as  $\mathbf{t}$ except that the leaf $v$ 
and its sibling leaf $v_1 \ldots v_{m-1} \bar v_m$ have both been removed and replaced by 
the single leaf $v_1 \ldots v_\ell$.
\end{notation}

\begin{remark}
If $\mathbf{t} = \mathbf R(z_1, \ldots, z_{n+1})$ for distinct $z_1, \ldots, z_{n+1}$ and $y_{n+1}$ is the
leaf of $\mathbf{t}$ corresponding to the input $z_{n+1}$, then
$\kappa(\mathbf{t}, y_{n+1}) = \mathbf R(z_1, \ldots, z_n)$.
\end{remark}
\begin{notation}
Consider $\bar {\mathbf{t}} \in \bar {\mathbb{S}}_{n+1}$ and let $v = v_1 \ldots v_m$ be a leaf of $\bar {\mathbf{t}}$. 
Let \mbox{$\bar \kappa(\bar {\mathbf{t}},v) \in \bar {\mathbb{S}}_n$} be the tree
$(\bar {\mathbf{t}} \setminus \{w \in \bar {\mathbf{t}} : v_1 \ldots v_{m-1} < w\}) \cup \{w_1 \ldots w_{m-1} w_{m+1} \ldots w_p : w \in \bar {\mathbf{t}}, \, v_1 \ldots v_{m-1} \bar v_m \le w_1 \ldots w_p\}$.
\end{notation}
\begin{remark}
If $\bar {\mathbf{t}} = \Phi \circ \mathbf R(z_1, \ldots, z_{n+1})$ for distinct $z_1, \ldots, z_{n+1}$ and $y_{n+1}$ is the
leaf of $\bar {\mathbf{t}}$ corresponding to the input $z_{n+1}$, then
$\bar \kappa(\bar {\mathbf{t}}, y_{n+1}) = \Phi \circ \mathbf R(z_1, \ldots, z_n)$.
\end{remark}
\begin{remark}\label{R:Remy}
Note that $\bar \kappa(\bar {\mathbf{t}}, v)$ is the tree obtained from $\bar {\mathbf{t}}$ by deleting $v$ and its sibling and closing up the resulting
gap if there is one (there will be a gap if and only if the sibling of $v$ is not a leaf). This operation is the same as one that appears in the backward transition of the R\'emy chain, and indeed, as part (ii) of the next proposition shows, the common backward transition probabilities of the PATRICIA chains are the same as that of the R\'emy chain described at the beginning of 
Section~4 of \cite{MR3601650}.
\end{remark}

\begin{proposition}
\label{P:Markov_property}
\begin{itemize}
\item[(i)]
The process $(\prescript{\nu}{}{R}_n)_{n \in \mathbb{N}}$ is Markov.  
For $\mathbf{s} \in \mathbb{S}_n$ such that $\mathbb{P}\{\prescript{\nu}{}{R}_n = \mathbf{s}\} > 0$ and
$\mathbf{t} \in \mathbb{S}_{n+1}$ such that $\mathbb{P}\{\prescript{\nu}{}{R}_{n+1} = \mathbf{t}\} > 0$
the associated backward transition probability is
\[
\mathbb{P}\{\prescript{\nu}{}{R}_n = \mathbf{s} \, | \, \prescript{\nu}{}{R}_{n+1} = \mathbf{t}\}
=
\begin{cases}
\frac{1}{n+1},& \text{if $\mathbf{s} = \kappa(\mathbf{t},v)$ for some $v \in \mathbf{L}(\mathbf{t})$,} \\
0,& \text{otherwise.} \\
\end{cases}
\]
\item[(ii)]
The process $(\prescript{\nu}{}{\bar R}_n)_{n \in \mathbb{N}}$ is Markov.  
For $\bar {\mathbf{s}} \in \bar {\mathbb{S}}_n$ such that $\mathbb{P}\{\prescript{\nu}{}{\bar R}_n = \bar {\mathbf{s}}\} > 0$ and
$\bar {\mathbf{t}} \in \bar {\mathbb{S}}_{n+1}$ such that $\mathbb{P}\{\prescript{\nu}{}{\bar R}_{n+1} = \bar {\mathbf{t}}\} > 0$
the associated backward transition probability is
\[
\mathbb{P}\{\prescript{\nu}{}{\bar R}_n = \bar {\mathbf{s}} \, | \, \prescript{\nu}{}{\bar R}_{n+1} = \bar {\mathbf{t}}\}
=
\begin{cases}
\frac{1}{n+1},& \text{if } \bar {\mathbf{s}} = \bar {\kappa}(\bar {\mathbf{t}},v) \text{ for some } v \in \mathbf{L}(\bar {\mathbf{t}}), \\
0,& \text{otherwise.} \\
\end{cases}
\]
\end{itemize}
\end{proposition}

\begin{proof}
Suppose that $\mathbf{s} \in \mathbb{S}_n$ and $\mathbf{t} \in \mathbb{S}_{n+1}$ is such that $\mathbb{P}\{\prescript{\nu}{}{R}_{n+1} = \mathbf{t}\} > 0$.
It is clear from Remark~\ref{R:permutation_invariance} and the assumption that $(Z_n)_{n \in \mathbb{N}}$ are independent that 
$\mathbb{P}\{\prescript{\nu}{}{R}_n = \mathbf{s} \, | \, \prescript{\nu}{}{R}_{n+1} = \mathbf{t}, \prescript{\nu}{}{R}_{n+2}, \ldots\}$
is $\frac{1}{n+1}$ if $\mathbf{s} = \kappa(\mathbf{t},v)$ for some $v \in \mathbf{L}(\mathbf{t})$ and $0$ otherwise.  In particular, the
time-reversal of $(\prescript{\nu}{}{R}_n)_{n \in \mathbb{N}}$ is Markovian and hence the same is true of $(\prescript{\nu}{}{R}_n)_{n \in \mathbb{N}}$
itself.  This establishes (i).

The proof of (ii) is similar.
\end{proof}
\begin{notation}
For $y \in \{0,1\}^\star$, let 
$\tau(y) := \{z \in \{0,1\}^\infty : y < z\}$;
 that is, $\tau(y)$ is the set of infinite extensions of the finite
word $y$. 
\end{notation}
\begin{remark}
An alternative route to establishing the Markov property of $(\prescript{\nu}{}{R}_n)_{n \in \mathbb{N}}$
is to note that the random set $\{Z_1, \ldots, Z_n\}$ is conditionally independent of
$\prescript{\nu}{}{R}_1, \ldots, \prescript{\nu}{}{R}_{n-1}$ given $\prescript{\nu}{}{R}_n$.
The conditional distribution of the random set $\{Z_1, \ldots, Z_n\}$ given
$\{\prescript{\nu}{}{R}_n = \mathbf{s}\}$, where $\mathbf{L}(\mathbf{s}) = \{y_1, \ldots, y_n\}$, is that
of the random set $\{\hat Z_1, \ldots, \hat Z_n\}$, where $\hat Z_1, \ldots, \hat Z_n$ are
independent and $\hat Z_k$ is distributed according to $\nu(\cdot \cap \tau(y_k)) / \nu(\tau(y_k))$
(that is, according to $\nu$ conditioned on $\tau(y_k)$) for $1 \le k \le n$.  Thus,
\[
\mathbb{P}\{\prescript{\nu}{}{R}_{n+1} = \mathbf{t} \, | \, \prescript{\nu}{}{R}_1, \ldots, \prescript{\nu}{}{R}_{n-1}, \prescript{\nu}{}{R}_n = \mathbf{s}\}
= \mathbb{P}\{\mathbf R(\hat Z_1, \ldots, \hat Z_n, Z_{n+1}) = \mathbf{t}\},
\]
where $\hat Z_1, \ldots, \hat Z_n$ are constructed to be independent of $Z_{n+1}$.
This observation leads
readily to an explicit calculation of the (forward) transition probabilities of $(\prescript{\nu}{}{R}_n)_{n \in \mathbb{N}}$
-- see Section~2 of \cite{MR3734107}.  The backward transition probabilities may then be derived as in
Section~3 of \cite{MR3734107}, giving a complete proof of part (i) of Proposition~\ref{P:Markov_property}.
\end{remark}

\begin{remark}
\label{R:Dynkin}
Once part (i) of Proposition~\ref{P:Markov_property} is known, an alternative derivation of part (ii) is to use the
Markov property of the time-reversal of $(\prescript{\nu}{}{R}_n)_{n \in \mathbb{N}}$ and apply Dynkin's classical
criterion   for the composition of a function with a Markov process to be Markovian (see, for instance, p. 575 of \cite{MR624684}).  Specifically,
one checks that if $\bar {\mathbf{t}} \in \bar {\mathbb{S}}_{n+1}$ is such that $\mathbb{P}\{\prescript{\nu}{}{\bar R}_{n+1} = \bar {\mathbf{t}}\} > 0$,
 $\mathbf{t} \in \mathbb{S}_{n+1}$ is such that $\mathbb{P}\{\prescript{\nu}{}{R}_{n+1} = \mathbf{t}\} > 0$ and
$\bar {\mathbf{t}} = \Phi(\mathbf{t})$, and 
$\bar {\mathbf{s}} \in \mathbb{S}_n$, then 
$\mathbb{P}\{\Phi(\prescript{\nu}{}{R}_n) = \bar {\mathbf{s}} \, | \, \prescript{\nu}{}{R}_{n+1} = \mathbf{t}\}$
is $\frac{1}{n+1}$ for $\bar {\mathbf{s}}$ of the form $\bar \kappa(\bar {\mathbf{t}}, v)$ when $v \in \mathbf{L}(\bar {\mathbf{t}})$ and
$0$ otherwise.  The fact that this conditional probability is the same for all $\mathbf{t} \in \mathbb{S}_{n+1}$ such that 
$\mathbb{P}\{\prescript{\nu}{}{R}_{n+1} = \mathbf{t}\} > 0$ and
$\bar {\mathbf{t}} = \Phi(\mathbf{t})$ gives that the time-reversal of $(\prescript{\nu}{}{\bar R}_n)_{n \in \mathbb{N}}$
is Markov and hence the same is true of the process $(\prescript{\nu}{}{\bar R}_n)_{n \in \mathbb{N}}$ itself.
Moreover, this same observation leads readily to the claimed backward transition probabilities.
\end{remark}

\begin{remark}  \label{R:Dynkin_fail}
Interestingly, Dynkin's criterion does not hold for the forwards in time processes
$(\prescript{\nu}{}{R}_n)_{n \in \mathbb{N}}$ and the function $\Phi$, as the following
example with $\nu~=~\gamma$ and $R_n:= \prescript{\gamma}{}{R}_n$, $n \in \mathbb{N}$, shows.
 Define trees $\mathbf{s}', \mathbf{s}'', \bar {\mathbf{t}}$ by
 $\mathbf{L}(\mathbf{s}') = \{00, 01, 1\}$
 $\mathbf{L}(\mathbf{s}'') = \{000, 001, 1\}$,
 and
 $\mathbf{L}(\bar {\mathbf{t}}) = \{000, 001, 01, 1\}$ (see Figure~\ref{Dynkin}).
 Note that $\Phi(\mathbf{s}') = \Phi(\mathbf{s}'') = \mathbf{s}'$.
 If $R_3 = \mathbf{s}'$, then the only way that $\Phi(R_4)$ can be $\bar {\mathbf{t}}$ is if
 $Z_4$ is of the form $00\ldots$; thus,
 $\mathbb{P}\{\Phi(R_4) = \bar {\mathbf{t}} \, | \, R_3 = \mathbf{s}'\} = \frac{1}{4}$.
 Similarly, if $R_3 = \mathbf{s}''$, then the only way that $\Phi(R_4)$ can be $\bar {\mathbf{t}}$ is if
 $Z_4$ is of the form $000\ldots$ or $01\ldots$; thus
 $\mathbb{P}\{\Phi(R_4) = \bar {\mathbf{t}} \, | \, R_3 = \mathbf{s}''\} = \frac{1}{4} + \frac{1}{8}$.
Since these two conditional probabilities are different, 
Dynkin's criterion does not hold.
\end{remark}
\begin{figure}
\centering
\includegraphics[height=3cm]{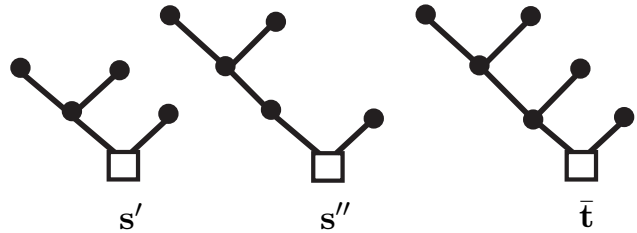}
\caption
{Illustration of Remark~\ref{R:Dynkin_fail}:   The trees $\mathbf s'$ and  $\mathbf s''$ have the same PATRICIA contraction. However, there are more possibilites (and a higher probability)  for a radix sort transition from $\mathbf s'$ than from $\mathbf s''$ to lead to a tree whose PATRICIA contraction is~$\bar {\mathbf t}$. 
}
\label{Dynkin}
\end{figure}

As observed in \cite{MR3734107}, different choices of $\nu$ result in different distributions of the radix sort chains $(\prescript{\nu}{}{R}_n)$. By way of contrast,
two different diffuse probability measures $\nu'$ and $\nu''$ on $\{0,1\}^\infty$ may result in one and the same distribution of the PATRICIA chains 
$(\prescript{\nu'}{}{\bar R}_n)$ and $(\prescript{\nu''}{}{\bar R}_n)$.
As a trivial  example, consider $\nu' = \gamma$ and $\nu'' = \delta_0 \otimes \gamma$ -- we may couple
the two radix sort chains together so that $\prescript{\nu''}{}{R}_n = 0 \prescript{\nu'}{}{R}_n$
for all $n \in \mathbb{N}$, in which case $\Phi(\prescript{\nu'}{}{\bar R}_n) = \Phi(\prescript{\nu''}{}{\bar R}_n)$
for all $n \in \mathbb{N}$.  

\begin{definition}
Declare that two probability measures $\nu'$ and $\nu''$ are {\em PATRICIA--equivalent}
when $(\prescript{\nu'}{}{\bar R}_n)_{n \in \mathbb{N}}$
and $(\prescript{\nu''}{}{\bar R}_n)_{n \in \mathbb{N}}$ have the same distribution. 
\end{definition}

For later use in Example~\ref{E:zig-zag_path} below we   show the following.

\begin{lemma}
\label{L:PATRICIA_fills_out}
For any diffuse probability measure $\nu$ on $\{0,1\}^\infty$ there exists a probability measure $\bar \nu$ on $\{0,1\}^\infty$ which is PATRICIA--equivalent  to $\nu$ and assigns strictly positive probability to $\tau(y)$ for every $y\in \{0,1\}^\star$.
\end{lemma}
\begin{proof}
We can write $\nu$ as  
$\delta_{u_1} \otimes \cdots \otimes \delta_{u_k} \otimes \nu^{\emptyset}$ 
for some finite, possibly empty, maximal sequence $u_1 \ldots u_k$
and some probability measure $\nu^{\emptyset}$ on $\{0,1\}^\infty$. Set $\nu_1 := \nu^{\emptyset}$.  
Note by the assumed maximality of $u_1 \ldots u_k$ that $\nu_1(\tau(0)) > 0$ and $\nu_1(\tau(1)) > 0$.

Suppose for some $n\in \mathbb{N}$ that we have built diffuse probability measures
$\nu_1, \ldots, \nu_n$ on $\{0,1\}^\infty$ such that:
\begin{itemize}
\item
$\nu_1, \ldots, \nu_n$ are
PATRICIA--equivalent to $\nu$, 
\item 
$\nu_n(\tau(y)) > 0$
for every $y \in \{0,1\}^n$,
\item
for $1 \le m < n$ the probability measures $\nu_m$ and $\nu_n$ agree on sets of the form 
$\tau(y)$, $y \in \{0,1\}^m$.
\end{itemize} 
 
Conditioning the probability measure
$\nu_n$ on the set $\tau(y)$, $y = y_1 \ldots y_n \in \{0,1\}^n$,  gives a probability measure of
 the form 
$\delta_{y_1} \otimes \cdots \otimes \delta_{y_n} \otimes \delta_{u_1} \otimes \cdots \otimes \delta_{u_k} \otimes \nu^y$ 
for some finite, possibly empty, maximal sequence $u_1 \ldots u_k$ 
(we re-use notation from above and our notation doesn't record the dependence of this
sequence on $y$) and some probability measure $\nu^y$ on $\{0,1\}^\infty$.  Note that 
$\nu^y(\tau(y0)) > 0$ and $\nu^y(\tau(y1)) > 0$. Put  
$\nu_{n+1} := \sum_{y \in \{0,1\}^n} \nu_n(\tau(y)) \delta_{y_1} \otimes \cdots \otimes \delta_{y_n} \otimes \nu^y$.  
It is clear that the probability measures $\nu_{n+1}$ and $\nu$ are PATRICIA--equivalent and that for 
$1 \le m < n+1$ the probability measures $\nu_m$ and $\nu_{n+1}$ agree on sets of the form 
$\tau(y)$, $y \in \{0,1\}^m$.

There is thus a unique diffuse probability measure 
$\bar \nu$ such that the probability measures $\nu_m$ and $\bar \nu$ agree on sets of the form 
$\tau(y)$, $y \in \{0,1\}^m$, for all $m \in \mathbb{N}$.  

It is not difficult to see that $\nu$ and $\bar \nu$  are PATRICIA--equivalent:  
we can couple together two i.i.d. sequences of inputs
distributed according to $\nu$ and $\bar \nu$ so that for each $n \in \mathbb{N}$
the tree $\prescript{\bar \nu}{}{R}_n$ is obtained from the tree $\prescript{\nu}{}{R}_n$
by a deterministic operation that removes certain vertices with out-degree $1$,
and hence $\Phi(\prescript{\bar \nu}{}{R}_n) = \Phi(\prescript{\nu}{}{R}_n)$ for all $n \in \mathbb{N}$.
\end{proof}

\begin{remark}
We leave to the reader the proof of the fact (which we shall not use) that if $\nu'$ and $\nu''$ are 
PATRICIA--equivalent, then, in the notation of Lemma~\ref{L:PATRICIA_fills_out},
$\bar {\nu'} = \bar {\nu''}$.
\end{remark}

It follows from Lemma~\ref{L:PATRICIA_fills_out} that for every diffuse probability measure 
$\nu$ on $\{0,1\}^\infty$ almost surely the successive states of the 
PATRICIA chain  $(\prescript{\nu}{}{\bar R}_n)_{n \in \mathbb{N}}$
``fill out'' the space $\{0,1\}^\star$ -- see Corollary~\ref{Cor:PATRICIA_fills_out}.  
This observation provides a useful
sufficient condition for determining when an infinite PATRICIA bridge is
of the form $(\prescript{\nu}{}{\bar R}_n)_{n \in \mathbb{N}}$ for some diffuse
probability measure $\nu$ on $\{0,1\}^\infty$ -- see Example~\ref{E:zig-zag_path}.

\begin{corollary}
\label{Cor:PATRICIA_fills_out}
For any diffuse probability measure $\nu$ on $\{0,1\}^\infty$  
we have $\bigcup_{m\in \mathbb N} \bigcap_{n \ge m} \prescript{\nu}{}{\bar R}_n  = \{0,1\}^\star$
almost surely.
\end{corollary}

\begin{proof}
Let $\bar \nu$ be the diffuse probability measure constructed in Lemma~\ref{L:PATRICIA_fills_out}
that is PATRICIA-equivalent to $\nu$ and satisfies $\bar \nu(\tau(y)) > 0$ for all $y\in \{0,1\}^\star$. 
Almost surely, for any $k \in \mathbb{N}$ there exists an $N \in \mathbb{N}$ such that 
for all $n \ge N$ and for all $y \in \{0,1\}^k$ there
is an $m \in [n]$ with $Z_m \in \tau(y)$.  Therefore, almost surely,
$\prescript{\bar \nu}{}{R}_n \supseteq \{\emptyset\} \cup \bigsqcup_{\ell=1}^k \{0,1\}^\ell$ 
for $n \ge N$.  Since $(\prescript{\nu}{}{\bar R}_n)_{n \in \mathbb{N}}$ 
and $(\prescript{\bar \nu}{}{\bar R}_n)_{n \in \mathbb{N}}$ have the same distribution,
this establishes the result.
\end{proof}

\section{Infinite bridges}
\label{S:bridges}
An {\em infinite bridge} for $(\prescript{\nu}{}{R}_n)_{n \in \mathbb{N}}$ (resp. $(\prescript{\nu}{}{\bar R}_n)_{n \in \mathbb{N}}$)
is a Markov chain $(\prescript{\nu}{}{R}_n^\infty)_{n \in \mathbb{N}}$ (resp. $(\prescript{\nu}{}{\bar R}_n^\infty)_{n \in \mathbb{N}}$)
such that $\prescript{\nu}{}{R}_1^\infty = \emptyset$ (resp. $\prescript{\nu}{}{\bar R}_1^\infty = \emptyset$) 
and
\[
\mathbb{P}\{\prescript{\nu}{}{R}_n^\infty = \mathbf{s} \, | \, \prescript{\nu}{}{R}_{n+1}^\infty = \mathbf{t}\}
=
 \mathbb{P}\{\prescript{\nu}{}{R}_n = \mathbf{s} \, | \, \prescript{\nu}{}{R}_{n+1} = \mathbf{t}\}
\]
for $\mathbf{s} \in \mathbb{S}_n$, $\mathbf{t} \in \mathbb{S}_{n+1}$, $n \in \mathbb{N}$
(resp. 
\[
\mathbb{P}\{\prescript{\nu}{}{\bar R}_n^\infty = \bar {\mathbf{s}} \, | \, \prescript{\nu}{}{\bar R}_{n+1}^\infty = \bar {\mathbf{t}}\}
= 
\mathbb{P}\{\prescript{\nu}{}{\bar R}_n = \bar {\mathbf{s}} \, | \, \prescript{\nu}{}{\bar R}_{n+1} = \bar {\mathbf{t}}\}
\]
for $\bar {\mathbf{s}} \in \bar S_n$, $\bar {\mathbf{t}} \in \bar {\mathbb{S}}_{n+1}$, $n \in \mathbb{N}$).  The reason we use this terminology
is the following.  A finite bridge  for $(\prescript{\nu}{}{R}_n)_{n \in \mathbb{N}}$ with end-point $\mathbf{t} \in \mathbb{S}_m$ for
some $m \in \mathbb{N}$ is the Markov process $(\prescript{\nu}{}{R}_n^\mathbf{t})_{n \in [m]}$
obtained by conditioning $(\prescript{\nu}{}{R}_n)_{n \in [m]}$ on the event $\{\prescript{\nu}{}{R}_m = \mathbf{t}\}$.
A finite bridge has the same backward transition probabilities as $(\prescript{\nu}{}{R}_n)_{n \in \mathbb{N}}$, and, writing $M(\mathbf{t})$ for
the number of leaves of the binary tree $\mathbf{t}$, one
way to produce an infinite bridge is to find a sequence of trees $(\mathbf{t}_k)_{k \in \mathbb{N}}$ such that $M(\mathbf{t}_k) \to \infty$
and the finite-dimensional distributions of   $(\prescript{\nu}{}{R}_n^{\mathbf{t}_k})_{n \in [M(\mathbf{t}_k)]}$ converge
as $k \to \infty$: the limiting finite-dimensional distributions will be consistent and hence define
a stochastic process $(\prescript{\nu}{}{R}_n^\infty)_{n \in \mathbb{N}}$ that is an infinite bridge for 
$(\prescript{\nu}{}{R}_n)_{n \in \mathbb{N}}$.  Similar remarks hold for the infinite bridges of $(\prescript{\nu}{}{\bar R}_n)_{n \in \mathbb{N}}$. 
We call the infinite bridges for the PATRICIA chains {\em infinite PATRICIA bridges}.

It follows from Proposition~\ref{P:Markov_property} that for an arbitrary diffuse probability measure~$\nu$ any infinite bridge for
$(\prescript{\nu}{}{R}_n)_{n \in \mathbb{N}}$ is an infinite bridge for $(R_n)_{n \in \mathbb{N}}$. Conversely, an infinite bridge
$(R_n^\infty)_{n \in \mathbb{N}}$ for $(R_n)_{n \in \mathbb{N}}$ is a bridge for $(\prescript{\nu}{}{R}_n)_{n \in \mathbb{N}}$ provided
$\{\mathbf{t} \in \mathbb{S}_n : \mathbb{P}\{R_n^\infty = \mathbf{t}\} > 0\} \subseteq \{\mathbf{t} \in \mathbb{S}_n : \mathbb{P}\{\prescript{\nu}{}{R}_n = \mathbf{t}\} > 0\}$
for all $n \in \mathbb{N}$.  We may therefore restrict our attention to infinite bridges for $(R_n)_{n \in \mathbb{N}}$.
Similar observations give that in investigating the infinite bridges of the processes $(\prescript{\nu}{}{\bar R}_n)_{n \in \mathbb{N}}$
we may restrict attention to those of $(\bar R_n)_{n \in \mathbb{N}}$.

It follows from the general theory of Doob--Martin compactifications that the family of distributions
of infinite bridges is a compact convex set, every element of which is unique mixture of extremal
distributions.  Moreover, every extremal infinite bridge is a limit of finite bridges 
(but the converse is not true in general).  Furthermore,
we see from Proposition~\ref{P:Markov_property} that each Markov process of the form
$(\prescript{\nu}{}{R}_n)_{n \in \mathbb{N}}$ (resp. $(\prescript{\nu}{}{\bar R}_n)_{n \in \mathbb{N}}$)
is an infinite bridge for $(R_n)_{n \in \mathbb{N}}$ (resp. $(\bar R_n)_{n \in \mathbb{N}}$) and, as we shall see below,
these infinite bridges are extremal.
The interesting questions
for the infinite bridges of $(R_n)_{n \in \mathbb{N}}$
and $(\bar R_n)_{n \in \mathbb{N}}$ are therefore:
\begin{itemize}
\item
What are all the extremal infinite bridges?
\item
Is every limit of finite bridges extremal?
\item
Is every extremal infinite bridge of the form $(\prescript{\nu}{}{R}_n)_{n \in \mathbb{N}}$ (resp. $(\prescript{\nu}{}{\bar R}_n)_{n \in \mathbb{N}}$)
for some $\nu$?
\end{itemize}

In the case of $(R_n)_{n \in \mathbb{N}}$, these questions were answered in \cite{MR3734107} as follows.

\newpage

\begin{theorem}
\label{T:radix_sort}
\begin{itemize}
\item[(i)]
An infinite bridge for $(R_n)_{n \in \mathbb{N}}$ is extremal if and only if it has an almost surely trivial tail $\sigma$-field.
\item[(ii)]
An infinite bridge for $(R_n)_{n \in \mathbb{N}}$ is extremal if and only if it is a weak limit of finite bridges.
\item[(iii)]
The extremal infinite bridges for $(R_n)_{n \in \mathbb{N}}$ coincide with collection of Markov processes of the form $(\prescript{\nu}{}{R}_n)_{n \in \mathbb{N}}$ for
some diffuse probability measure $\nu$.
\end{itemize}
\end{theorem}

By part (ii) of Proposition~\ref{P:Markov_property} the class of
infinite PATRICIA bridges and the class of infinite R\'emy bridges coincide. 
The next theorem shows that analogues of parts (i)
and (ii) Theorem~\ref{T:radix_sort} hold for the infinite PATRICIA bridges. 

\begin{theorem}
\label{T:PATRICIA_i_ii}
\begin{itemize}
\item[(i)]
An infinite bridge for $(\bar R_n)_{n \in \mathbb{N}}$ is extremal if and only if it has an almost surely trivial tail $\sigma$-field.
\item[(ii)]
An infinite bridge for $(\bar R_n)_{n \in \mathbb{N}}$ is extremal if and only if it is a weak limit of finite bridges.
\end{itemize}
\end{theorem}

\begin{proof}
See Proposition 5.19 and the proof of Corollary 5.21 in \cite{MR3601650}.
\end{proof}

\begin{remark}\label{R:EGW2corr} Let us comment briefly on the proof of Proposition 5.19 in \cite{MR3601650}.
Remark 5.20 in  in that paper says that this proposition can also be proved along the lines of Lemma 5.3 therein. As pointed out to us by Julian Gerstenberg and Rudolf Gr\"ubel, the argument in the proof of Lemma 5.3 and consequently Remark~5.20 is incorrect. However, this does not invalidate the proof of Proposition 5.19 given directly after the statement, as this proof is along completely different lines. Also, this does not invalidate any of the remainder of \cite{MR3601650} since Lemma 5.3 and Remark 5.20 of that paper are not used further therein.
\end{remark}

The following proposition gives a way of producing (extremal) infinite
PATRICIA bridges from (extremal) infinite radix sort bridges.

\begin{proposition}
\label{P:bridges_to_bridges}
If $(R_n^\infty)_{n \in \mathbb{N}}$ is an infinite bridge for $(R_n)_{n \in \mathbb{N}}$, then
$(\bar R_n^\infty)_{n \in \mathbb{N}}:= (\Phi(R_n^\infty))_{n \in \mathbb{N}}$ is an infinite bridge for
$(\bar R_n)_{n \in \mathbb{N}}$.
Moreover, if $(R_n^\infty)_{n \in \mathbb{N}}$ is extremal and thus of the form
$(\prescript{\nu}{}{R}_n)_{n \in \mathbb{N}}$ for some diffuse probability measure $\nu$, then
$(\bar R_n^\infty)_{n \in \mathbb{N}} = (\prescript{\nu}{}{\bar R}_n)_{n \in \mathbb{N}}$
is extremal.
\end{proposition}

\begin{proof}
Consider the first claim.  For $n \in \mathbb{N}$, $\mathbf{s} \in \mathbb{S}_n$, $\mathbf{t} \in \mathbb{S}_{n+1}$,
$\bar {\mathbf{s}} \in \bar {\mathbb{S}}_n$, and $\bar {\mathbf{t}} \in \bar {\mathbb{S}}_{n+1}$ put
$\Lambda_n(\mathbf{s}) := \mathbb{P}\{R_n^\infty = \mathbf{s}\}$, $\Lambda_{n+1}(\mathbf{t}) := \mathbb{P}\{R_{n+1}^\infty = \mathbf{t}\}$,
$Q_{n+1,n}(\mathbf{t}, \mathbf{s}) := \mathbb{P}\{R_n^\infty = \mathbf{s} \, | \, R_{n+1}^\infty = \mathbf{t}\} = \mathbb{P}\{R_n = \mathbf{s} \, | \, R_{n+1} = \mathbf{t}\}$,
$\bar \Lambda_n(\bar {\mathbf{s}}) := \mathbb{P}\{\bar R_n^\infty = \bar {\mathbf{s}}\}$, $\bar \Lambda_{n+1}(\bar {\mathbf{t}}) := \mathbb{P}\{\bar R_{n+1}^\infty = \bar {\mathbf{t}}\}$,
and
$\bar Q_{n+1,n}(\bar {\mathbf{t}}, \bar {\mathbf{s}}) := \mathbb{P}\{\bar R_n = \bar {\mathbf{s}} \, | \, \bar R_{n+1} = \bar {\mathbf{t}}\}$.
Because $(R_n^\infty)_{n \in \mathbb{N}}$ is an infinite bridge for $(R_n)_{n \in \mathbb{N}}$ we know that
\[
\sum_{\mathbf{t} \in \mathbb{S}_{n+1}} \Lambda_{n+1}(\mathbf{t}) Q_{n+1,n}(\mathbf{t}, \mathbf{s}) = \Lambda_n(\mathbf{s})
\]
and in order to show that $(\bar R_n^\infty)_{n \in \mathbb{N}}$ is an infinite bridge for $(\bar R_n)_{n \in \mathbb{N}}$
we need to show that
\[
\sum_{\bar {\mathbf{t}} \in \bar {\mathbb{S}}_{n+1}} \bar \Lambda_{n+1}(\bar {\mathbf{t}}) \bar Q_{n+1,n}(\bar {\mathbf{t}}, \bar {\mathbf{s}}) = \bar \Lambda_n(\bar {\mathbf{s}}).
\]
Now 
\[
\bar \Lambda_n(\bar {\mathbf{s}}) = \sum_{\mathbf{s} \in \mathbb{S}_n : \Phi(\mathbf{s}) = \bar {\mathbf{s}}} \Lambda_n(\mathbf{s}),
\]
and
\[
\bar \Lambda_{n+1}(\bar {\mathbf{t}}) = \sum_{\mathbf{t} \in \mathbb{S}_{n+1} : \Phi(\mathbf{t}) = \bar {\mathbf{t}}} \Lambda_{n+1}(\mathbf{t}).
\]
Furthermore, by Remark~\ref{R:Dynkin}
we have
\[
\bar Q_{n+1,n}(\bar {\mathbf{t}}, \bar{ \mathbf{s}}) = \sum_{\mathbf{s} \in \mathbb{S}_n : \Phi(\mathbf{s}) = \bar {\mathbf{s}}} Q_{n+1,n}(\mathbf{t}, \mathbf{s})
\]
when $\Phi(\mathbf{t}) = \bar {\mathbf{t}}$.  Therefore,
\[
\begin{split}
\sum_{\bar {\mathbf{t}} \in \bar {\mathbb{S}}_{n+1}} 
{\bar \Lambda}_{n+1}(\bar {\mathbf{t}}) {\bar Q}_{n+1,n}(\bar {\mathbf{t}}, \bar {\mathbf{s}}) 
 &=
\sum_{\bar {\mathbf{t}} \in \bar {\mathbb{S}}_{n+1}} \sum_{\mathbf{t} \in \mathbb{S}_{n+1} : \Phi(\mathbf{t}) = \bar {\mathbf{t}}} \Lambda_{n+1}(\mathbf{t}) \sum_{\mathbf{s} \in \mathbb{S}_n : \Phi(\mathbf{s}) = \bar {\mathbf{s}}} Q_{n+1,n}(\mathbf{t}, \mathbf{s}) \\
& =
\sum_{\mathbf{s} \in \mathbb{S}_n : \Phi(\mathbf{s}) = \bar {\mathbf{s}}} \sum_{\bar {\mathbf{t}} \in \bar {\mathbb{S}}_{n+1}} \sum_{\mathbf{t} \in \mathbb{S}_{n+1} : \Phi(\mathbf{t}) = \bar {\mathbf{t}}} \Lambda_{n+1}(\mathbf{t}) Q_{n+1,n}(\mathbf{t}, \mathbf{s}) \\ 
& =
\sum_{\mathbf{s} \in \mathbb{S}_n : \Phi(\mathbf{s}) = \bar {\mathbf{s}}} \sum_{\mathbf{t} \in \mathbb{S}_{n+1}} \Lambda_{n+1}(\mathbf{t}) Q_{n+1,n}(\mathbf{t}, \mathbf{s}) \\
& =
\sum_{\mathbf{s} \in \mathbb{S}_n : \Phi(\mathbf{s}) = \bar {\mathbf{s}}} \Lambda_n(\mathbf{s}) \\
& =
\bar \Lambda_n(\bar {\mathbf{s}}),  \\
\end{split}
\]
as required.

Turning to the second claim, suppose that the infinite bridge $(R_n^\infty)_{n \in \mathbb{N}}$ 
is extremal (and hence, by part (iii) of Theorem~\ref{T:radix_sort} is of the form
$(\prescript{\nu}{}{R}_n)_{n \in \mathbb{N}}$ for some diffuse probability measure $\nu$).  
By part (i) of Theorem~\ref{T:radix_sort} the tail $\sigma$-field of 
$(R_n^\infty)_{n \in \mathbb{N}}$ is almost surely trivial.  It follows that the tail $\sigma$-field of
$(\bar R_n^\infty)_{n \in \mathbb{N}}:= (\Phi(R_n^\infty))_{n \in \mathbb{N}}$ is almost surely trivial and hence,
by part (i) of Theorem~\ref{T:PATRICIA_i_ii}, the infinite bridge $(\bar R_n^\infty)_{n \in \mathbb{N}}$
is extremal. (Alternatively, it is immediate from
 Remark~\ref{R:permutation_invariance} and the Hewitt-Savage zero-one law
that $(\bar R_n^\infty)_{n \in \mathbb{N}} = (\prescript{\nu}{}{\bar R}_n)_{n \in \mathbb{N}}$ 
has an almost surely trivial tail $\sigma$-field.)
\end{proof}

\begin{remark}\label{R:ePb}
Comparing Theorem~\ref{T:radix_sort} with Theorem~\ref{T:PATRICIA_i_ii} and taking
Proposition~\ref{P:bridges_to_bridges} into account, it is natural to conjecture
that the analogue of part (iii) of Theorem~\ref{T:radix_sort} holds for
infinite PATRICIA bridges; that is, that the set of extremal PATRICIA bridges is 
exhausted by the PATRICIA chains $(\prescript{\nu}{}{\bar R}_n)_{n \in \mathbb{N}}$, 
with $\nu$ a diffuse probability measure on $\{0,1\}^\infty$.
This, however, is not the case, as the next example shows.
This example recalls the construction from Example~4.1 of \cite{MR3601650} 
of a specific extremal  infinite R\'emy (and hence PATRICIA) bridge based on a 
random zig-zag path with leaves attached to it. 
We will use Corollary~\ref{Cor:PATRICIA_fills_out} to conclude that this extremal
infinite PATRICIA is not of the form $(\prescript{\nu}{}{\bar R}_n)_{n \in \mathbb{N}}$ 
for any diffuse probability measure $\nu$ on $\{0,1\}^\infty$
A concrete description of the extremal infinite PATRICIA bridges comprises the remainder of the paper.
\end{remark}

\begin{example}
\label{E:zig-zag_path} 
Let $\epsilon_2, \epsilon_3, \ldots$ be a sequence of independent identically distributed $\{0,1\}$-valued
random variables with common distribution \mbox{$\mathbb{P}\{\epsilon_k = 0\}$} \mbox{$= \mathbb{P}\{\epsilon_k = 1\}$} $ = \frac{1}{2}$, $2 \le k < \infty$.
For $N \in \mathbb{N}$, define $(\bar R_n^N)_{n \in [N]}$ with $\bar R_n^N \in \bar {\mathbb{S}}_n$ 
for $n \in [N]$ by requiring that $(\bar R_n^N)_{n \in [N]}$ is Markov,
$\bar R_N^N$ has the same distribution as
$\{\emptyset\} \cup \bigcup_{k=2}^N \{\epsilon_2 \ldots \epsilon_k, \, \epsilon_2 \ldots \epsilon_{k-1} \bar\epsilon_k \}$,
and the backward transition probabilities of $(\bar R_n^N)_{n \in [N]}$ are the same as those
of $(\bar R_n)_{n \in \mathbb{N}}$. We may suppose that $\bar R_N^N
= \{\emptyset\} \cup \bigcup_{k=2}^N \{\epsilon_2 \ldots \epsilon_k, \, \epsilon_2 \ldots \epsilon_{k-1} \bar\epsilon_k \}$. It follows from the form of the backward transition probabilities that
for $M \in [N]$, $\bar R_M^N
=
\{\emptyset\} \cup \bigcup_{k=2}^M \{\epsilon_{I_2} \ldots \epsilon_{I_k}, \, \epsilon_{I_2} \ldots \epsilon_{k-1} \bar\epsilon_{I_k} \}$,
where $2 \le I_2 < I_3 < \ldots < I_M \le N$ is a certain uniform random subset of $\{2, \ldots, N\}$ of cardinality $M$ that is independent of
$\epsilon_2, \ldots, \epsilon_N$.   Thus $\bar R_M^N$ has the same distribution as $\bar R_M^M$ and hence, by Kolmogorov's extension theorem,
there exists a Markov process $(\bar R_n^\infty)_{n \in \mathbb{N}}$ such that $(\bar R_n^\infty)_{n \in [N]}$ has the same distribution as
$(\bar R_n^N)_{n \in [N]}$ for any $N \in \mathbb{N}$.  Therefore $(\bar R_n^\infty)_{n \in \mathbb{N}}$ is an infinite bridge
for $(\bar R_n)_{n \in \mathbb{N}}$.  

We can give a pathwise construction of $(\bar R_n^\infty)_{n \in \mathbb{N}}$ as follows.  Let $((Y_n, \eta_n))_{n = 2}^\infty$ be an infinite sequence
of independent identically distributed $[0,1] \times \{0,1\}$-valued random variables such that $Y_n$ has the uniform distribution on $[0,1]$,
$\mbox{$\mathbb{P}\{\eta_n = 0\}$} = \mathbb{P}\{\eta_n = 1\} = \frac{1}{2}$, and $Y_n$ and $\eta_n$ are independent.  For $2 \le n < \infty$, let $\sigma_n$ be the
permutation of $\{2, \ldots, n\}$ such that $Y_{\sigma_n(2)} < Y_{\sigma_n(3)} < \ldots < Y_{\sigma_n(n)}$. For $2 \le k \le n$, put
$\epsilon_{n,k} = \eta_{\sigma_n(k)}$.  Setting $\bar R_1^\infty = \emptyset$ and
$\bar R_n^\infty := \{\emptyset\} \cup \bigcup_{k=2}^n \{\epsilon_{n,2} \ldots \epsilon_{n,k}, \, \epsilon_{n,2} \ldots \epsilon_{n,(k-1)} \bar\epsilon_{n,k} \}$ for $n \ge 2$
produces a process with the desired distribution.
Note that if $\pi$ is a permutation of $\{2,3,\ldots\}$ that leaves every element of $\{N+1, N+2, \ldots\}$ fixed for some $N \ge 2$ and we replace
 $((Y_n, \eta_n))_{n = 2}^\infty$ by  $((Y_{\pi(n)}, \eta_{\pi(n)}))_{n = 2}^\infty$ in this construction, then the values of $\bar R_n^\infty$, $n \ge N$,
are left unchanged.  It therefore follows from the Hewitt-Savage zero-one law that the tail $\sigma$-field of $(\bar R_n^\infty)_{n \in \mathbb{N}}$ is $\mathbb{P}$-a.s. trivial
and thus, by part (i) of Theorem~\ref{T:PATRICIA_i_ii}, this process is an extremal infinite bridge for $(\bar R_n)_{n \in \mathbb{N}}$.

\medskip
\noindent
\begin{minipage}{5.2cm}
Now, for any $n \in \mathbb{N}$ the random tree $\bar R_n^\infty$ is a zig-zag path of length 
$n-1$ with leaves attached to it.  It follows that for any $m \in \mathbb{N}$ the random tree 
$\bigcap_{n \ge m}\bar R_n^\infty$ is a zig-zag path of length at most \mbox{$m-1$} with leaves 
attached to it
and hence $\bigcup_{m \in \mathbb{N}} \bigcap_{n \ge m}\bar R_n^\infty$ is a 
zig-zag path of possibly infinite length with leaves attached to it and is certainly not
all of $\{0,1\}^\star$.  
\end{minipage}
\qquad \quad 
\begin{minipage}{4.9cm}
\qquad \qquad   \includegraphics[height=3cm]{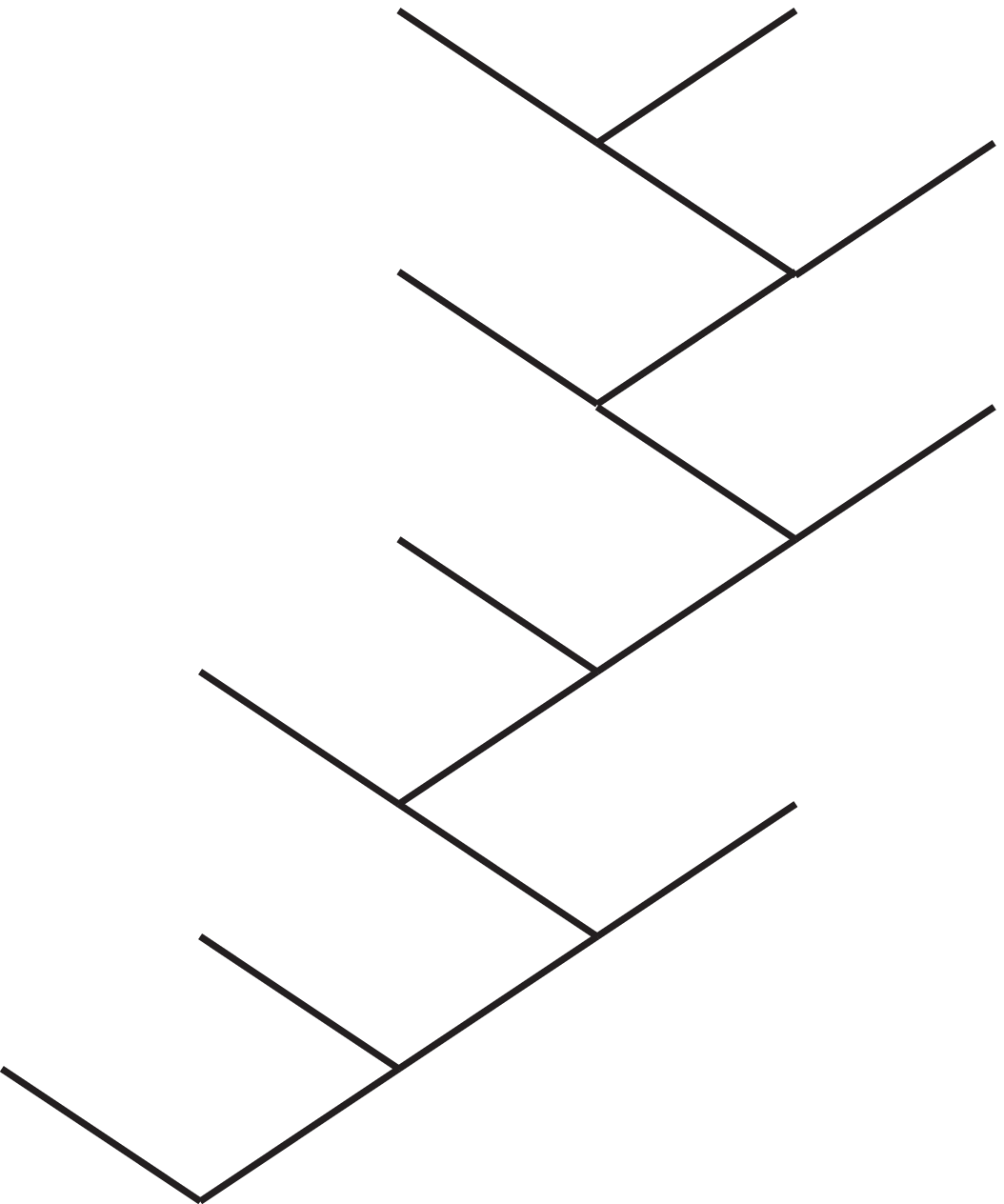}\\
 \textsc{Fig.~4.1} \,\, A realization  of  $\bar R_9^\infty$.
\end{minipage}

\vspace{0.1cm}
\noindent
It is thus clear from
Corollary~\ref{Cor:PATRICIA_fills_out} that the extremal infinite bridge  
$(\bar R_n^\infty)_{n \in \mathbb{N}}$ is not of the form 
$(\prescript{\nu}{}{\bar R}_n)_{n \in \mathbb{N}}$
for any diffuse probability measure~$\nu$.  

For the sake of completeness, we note that if we fix
$m \in \mathbb{N}$, then for $n \ge m$ the restriction of $\bar R_n^\infty$ to height $m-1$
is uniformly distributed over the set of zig-zag paths of length $m-1$ with leaves attached.
 Moreover, if $m \le N_1 < N_2 < \ldots < N_M$ for some fixed $M \in \mathbb{N}$, then the
restrictions of $\bar R_{N_1}^\infty, \bar R_{N_2}^\infty, \ldots, 
\bar R_{N_M}^\infty$ to height $m-1$ become asymptotically independent as 
$N_2-N_1, \ldots, N_M-N_{M-1} \to \infty$.  It follows that almost surely 
$\bigcup_{m \in \mathbb{N}} \bigcap_{n \ge m}\bar R_n^\infty$ is just the set $\{\emptyset, 0, 1\}$.
Incidentally, from these same observations it is clear that 
$\bigcap_{m \in \mathbb{N}} \bigcup_{n \ge m}\bar R_n^\infty = \{0,1\}^\star$ almost surely. This latter property would therefore not suffice to discriminate $(\bar R_n^\infty)_{n \in \mathbb{N}}$ from the extremal infinite bridges of the form $(\prescript{\nu}{}{\bar R}_n)_{n \in \mathbb{N}}$.
\end{example}
Our next goal is to introduce a class of extremal infinite bridges for $(\bar R_n)_{n \in \mathbb{N}}$
that subsumes both the class of Markov chains of the form $(\prescript{\nu}{}{\bar R}_n)_{n \in \mathbb{N}}$ and
the Markov chain in Example~\ref{E:zig-zag_path}. 
In fact, using Proposition~\ref{P:Markov_property} and Remark~\ref{R:Remy} 
(showing that the common backward transitions of the PATRICIA chains coincide with those of the R\'emy chain)  
we will be able to obtain a representation for {\em all}  the extremal infinite PATRICIA bridges. 
This will be achieved through the concept of a
{\em didendritic system} introduced in   Definition~5.8 of \cite{MR3601650}  
and revisited in Section~\ref{S:DDS}.

Before moving on to this characterization of all the extremal
infinite PATRICA bridges, we include the following section which
gives some idea of the divergent sample path behavior displayed by different
infinite bridges.


\section{Heights of trees in an infinite PATRICIA bridge}
\label{S:heights}

Given $\bar{\mathbf{t}} \in \bar{\mathbb{S}}_n$ for some $n \ge 0$ with leaves $y_1, \ldots, y_n$
let $\mathrm{ht}(\bar{\mathbf{t}})$ denote the height of $\bar{\mathbf{t}}$; that is,
$\mathrm{ht}(\bar{\mathbf{t}}) := \max_{1 \le k \le n} |y_k|$.  It is shown in 
\cite{MR1151362} that
\[
\lim_{n \to \infty} (\mathrm{ht}(\bar R_n) - \log_2 n)/\sqrt{2 \log_2 n} = 1
\]
almost surely so that, in particular,
\[
\lim_{n \to \infty} \mathrm{ht}(\bar R_n) / \log_2 n = 1
\]
almost surely.

It is of interest to compare this asymptotic behavior with that of other
infinite PATRICIA bridges.  The infinite PATRICIA bridge $(\bar R_n^\infty)_{n \in \mathbb{N}}$
of Example~\ref{E:zig-zag_path}
clearly has $\mathrm{ht}(\bar R_n^\infty) = n-1$ for all $n \in \mathbb{N}$.  On the other
hand if $(\bar R_n^\infty)_{n \in \mathbb{N}}$ is the R\'emy tree growth chain
(which we have observed is an infinite PATRICIA bridge),
then it follows from the results of \cite{MR2042386} that
\[
\lim_{n \to \infty} \mathrm{ht}(\bar R_n^\infty) / \sqrt{n}
\]
exists and is strictly positive almost surely.

As one last example, it has been pointed out to us by Ralph Neininger that
if $\nu = \bigotimes_{j =1}^\infty \mathrm{Ber}\left(\frac{1}{j+1}\right)$,
where $\mathrm{Ber}(p)$ is the Bernoulli probability measure on $\{0,1\}$ that
assigns mass $1-p$ to $0$ and mass $p$ to $1$, then
\[
\limsup_{n \to \infty} \mathrm{ht}(\prescript{\nu}{}{\bar R}_n)
/\sqrt{2n/\log n} \ge 1
\]
almost surely.  For the sake of completeness, we sketch Neininger's argument.

Suppose that $Z_1, Z_2, \ldots$ is an infinite i.i.d. sequence of random
elements of $\{0,1\}^\infty$ where $Z_n = (Z_{n,1}, Z_{n,2}, \ldots)$, $n \in \mathbb{N}$,
with $Z_{n,1}, Z_{n,2}, \ldots$ independent, $\mathbb{P}\{Z_{n,j} = 0\} = \frac{j}{j+1}$,
and $\mathbb{P}\{Z_{n,j} = 1\} = \frac{1}{j+1}$, $j \in \mathbb{N}$.  
Let $\prescript{\nu}{}{\bar R}_n$, $n \in \mathbb{N}$, be the PATRICIA tree constructed
from the inputs $Z_1, \ldots, Z_n$.

Consider the event
\[
A_{n,t} := \bigcap_{\ell=1}^t \bigcup_{m=1}^n \{(Z_{m,1}, \ldots, Z_{m,\ell}) = (0,\ldots,0,1)\}.
\]
Observe that on the event $A_{n,t}$ the tree $\prescript{\nu}{}{\bar R}_n$
contains the vertices $\emptyset, 1, 01, 001, \ldots, 0 \ldots 01$, where the last
sequence has $t-1$ zeros.  Consequently, 
$\mathrm{ht}(\prescript{\nu}{}{\bar R}_n) \ge t$ on the event $A_{n,t}$.
It therefore suffices to show that $\lim_{n \to \infty} \mathbb{P}(A_{n,t(n)}) = 1$,
where $t(n) = \lfloor \sqrt{2n/\log n} -1 \rfloor$. 

Now
\[
\begin{split}
1 - \mathbb{P}(A_{n,t}) 
&= 
\mathbb{P}(A_{n,t}^c) \\
&=
\mathbb{P}\left(\bigcup_{\ell=1}^t \bigcap_{m=1}^n \{(Z_{m,1}, \ldots, Z_{m,\ell}) \ne (0,\ldots,0,1)\}\right) \\
& \le
\sum_{\ell=1}^t \mathbb{P}\left(\bigcap_{m=1}^n \{(Z_{m,1}, \ldots, Z_{m,\ell}) \ne (0,\ldots,0,1)\}\right) \\
& =
\sum_{\ell=1}^t \left(1 - \frac{1}{\ell+1} \prod_{j=1}^{\ell-1} \frac{j}{j+1} \right)^n \\
&=
\sum_{\ell=1}^t \left(1 - \frac{1}{\ell(\ell+1)}\right)^n 
\\
&
 \le
t \left(1 - \frac{1}{t(t+1)}\right)^n \\
& \le
(t+1) \left(1 - \frac{1}{(t+1)^2}\right)^n  \le
(t+1) \exp\left(-\frac{n}{(t+1)^2}\right). 
\end{split}
\]
Thus
\[
\begin{split}
1 - \mathbb{P}(A_{n,t(n)})
& \le
\sqrt{2n/\log n} \exp\left(-\frac{n}{(\sqrt{2n/\log n})^2}\right) \\
& =
\sqrt{2n/\log n} \exp\left(-\frac{\log n}{2}\right) \\
& =
\sqrt{2 /\log n} \\
& \to 0 \quad \text{as $n \to \infty$}, \\
\end{split}
\]
as required.

\section{Didendritic systems}
\label{S:DDS}
We now prepare some notation and definitions that serve to introduce the concept of a didendritic system. As announced at the end of Section~\ref{S:bridges}, this concept will help to obtain a representation for {\em all}  the extremal infinite PATRICIA bridges. 

With the partial order $<$ defined on the set $\{0,1\}^\star \sqcup \{0,1\}^\infty$, we define, for any two elements  
$u,v$ of this set, the {\em most recent predecessor} $u \wedge v$ of   
$u$ and~$v$. More precisely, for each such  $u,v$ there  is a unique $w \in \{0,1\}^\star \sqcup \{0,1\}^\infty$
with the properties $w \le u$, $w \le v$, and if $x$ is any element of $\{0,1\}^\star \sqcup \{0,1\}^\infty$
with $x \le u$ and $x \le v$, then $x \le w$.  Denote this $w$ by $u \wedge v$.  For example,
if $u = u_1 \ldots u_m$ and $v = v_1 \ldots v_n$ are elements of $\{0,1\}^\star$, then
$u \wedge v = u_1 \ldots u_p = v_1 \ldots v_p$, where 
$p:= \max\{q  :1 \le q \le m \wedge n, \, u_1 \ldots u_q = v_1 \ldots v_q\}$ and the maximum of
the empty set is defined to be zero.

We define two further partial orders $<_L$ and $<_R$ on $\{0,1\}^\star$ by declaring that
$u_1 \ldots u_m <_L v_1 \ldots v_n$ 
(resp. $u_1 \ldots u_m <_R v_1 \ldots v_n$) if
$m < n$ and $v_1 \ldots v_{m+1} = u_1 \ldots u_m 0$ 
(resp. $m < n$ and $v_1 \ldots v_{m+1} = u_1 \ldots u_m 1$),

\begin{definition}
\label{D:finite_didendritic}
Given $\bar {\mathbf{t}} \in \bar {\mathbb{S}}_n$, consider a bijective labeling of the leaves of~$\bar {\mathbf{t}}$ 
by a set $\mathcal{N}$ with $\# \mathcal{N} = n$ (that is, each leaf receives a distinct label).  
Use the notation $\tilde {\mathbf{t}}$ to denote the labeled object.  
Define an equivalence relation $\equiv$ on $\mathcal{N} \times \mathcal{N}$ by declaring that
$(g,h)$ and $(i,j)$ are equivalent if $g,h,i,j$  label respectively
leaves $u,v,w,x$ such that $u \wedge v = w \wedge x$. 
Denote by $\langle i,j \rangle$ the equivalence class containing $(i,j) \in \mathcal{N} \times \mathcal{N}$.
Note that $(i,i)$ is the only pair in the equivalence class $\langle i, i \rangle$ so we
will usually denote this equivalence class more simply by $i$.  If $k,\ell$ label respectively the
leaves $y,z$, then label the vertex $y \wedge z$ with $\langle k, \ell \rangle$.

With a slight abuse of notation, 
define a partial order $<_L$ (resp. $<_R$) on $\{\langle i,j \rangle: i,j \in \mathcal{N}\}$ by
declaring that $\langle g,h \rangle <_L \langle i,j \rangle$ if $\langle g,h \rangle$
 labels a vertex $u$ and  $ \langle i,j \rangle$ labels a vertex $v$ such that
$u <_L v$ (resp. $u <_R v$). Similarly, define a third partial order $<$ by declaring that
$\langle g,h \rangle  < \langle i,j \rangle$
if $\langle g,h \rangle$ labels a vertex $u$ and $\langle i,j \rangle$ labels a vertex $v$
with $u < v$.
\end{definition}
The equivalence relation $\equiv$ and the partial orders $<_L$, $<_R$, and $<$ 
have the following properties.
\begin{itemize}
\item[(A)]
For $i,j \in \mathcal N$, $(i,j) \equiv (j,i)$.
\item[(B)] 
For distinct $i,j \in \mathcal N$, either $\langle i,j\rangle <_L  \langle i,i\rangle$ and $\langle i,j\rangle <_R  \langle j,j\rangle$, 
or $\langle i,j\rangle <_R  \langle i,i\rangle$ and $\langle i,j\rangle <_L  \langle j,j\rangle$.
\item[(C)] {\em ``Triplet property''}  For distinct $i, j, k$, exactly one of \\
\phantom{AAAAA}$\langle i,j\rangle = \langle i,k\rangle < \langle j,k\rangle$
\\
\phantom{AAAAA}$\langle j,k\rangle = \langle j,i\rangle < \langle k,i \rangle$
\\
or\\
\phantom{AAAAA}$\langle k,i\rangle = \langle k,j\rangle < \langle i,j \rangle$ \\
is valid.
\item[(D)] For $i,j,k,\ell \in \mathcal N$, at most one of the relations $\langle i,j \rangle <_L \langle k,\ell \rangle$ and $\langle i,j \rangle <_R \langle k,\ell \rangle$ can hold and 
$\langle i,j \rangle < \langle k,\ell \rangle$ if and only if either $\langle i,j \rangle <_L\langle k,\ell \rangle$ or $\langle i,j \rangle <_R\langle k,\ell \rangle$.
\item[(E)] Fix $f,g,h,i,j,k \in \mathcal N$. 
If
$\langle f,g \rangle <_L \langle h,i \rangle  < \langle j,k \rangle$, then
$\langle f,g \rangle <_L \langle j,k \rangle$.
Similarly, if
$\langle f,g \rangle <_R \langle h,i \rangle  < \langle j,k \rangle$, then
$\langle f,g \rangle <_R \langle j,k \rangle$.
\end{itemize}

\begin{definition}
\label{D:didendritic}
A {\em didendritic system} $\mathcal D = (\mathcal N, \equiv, \langle \cdot, \cdot \rangle, <_L, <_R, <)$ with the 
non-empty (possibly infinite) {\em label set} $\mathcal N$
is the set $\mathcal N \times \mathcal N$
equipped with an equivalence relation $\equiv$,
equivalence classes $\langle \cdot, \cdot \rangle$,
and partial orders $<_L$, $<_R$ and $<$ on the set of equivalence classes
such that the above stated properties (A)-(E) hold. 
\end{definition}

\begin{remark}
We show in Proposition~\ref{DDS_tree} that any didendritic system with a finite label set may be thought of as a leaf-labeled full binary tree. This claim was made in Remark~5.10 of \cite{MR3601650}. The axioms in Definition~\ref{D:didendritic} differ from those in  \cite{MR3601650}: the former are equivalent to the latter plus the triplet property (C). The Example~\ref{E:not_a_DDS} shows that the conclusion of Proposition~\ref{DDS_tree} does not hold under the axioms of \cite{MR3601650}, contrary to the assertion of 
Remark~5.10 of \cite{MR3601650}.  However, the omission of axiom (C) in \cite{MR3601650} is not critical: the development there depends solely
on the conclusion of Remark~5.10 therein rather than the specific axioms that one uses to characterize a didendritic system.  In brief, the arguments in
\cite{MR3601650} become correct once one adds axiom (C) to the definition there of a didendritic system.
\end{remark}

\begin{example}
\label{E:not_a_DDS}
Consider the equivalence relation on $[3] \times [3]$ defined by 
$$(h,i)\equiv (j,k) \mbox{ if an only if } h=j \mbox{ and } i=k, \mbox{ or } h=k \mbox{ and } i=j.$$
For $(i, j) \in \{(1,2),(1,3), (2,3)\} $ say that $\langle i,j\rangle <_L \langle i,i\rangle$ and $\langle i,j\rangle <_R \langle j,j\rangle$. Moreover, say that $\langle h,i \rangle < \langle j,k \rangle$ if $\langle h,i \rangle <_L \langle j,k \rangle$ or $\langle h,i \rangle <_R \langle j,k \rangle$. Then $([3], \equiv, <_L, <_R,<)$ meets the axioms of Definition 5.8 of \cite{MR3601650}. However,  this system  does not correspond to a binary tree with 3~leaves, since this would require exactly $5$ equivalence classes (each one corresponding to a vertex of the tree), whereas this system has $3+3 = 6$ equivalence classes.
\end{example}

\begin{lemma}
\label{ordered_triplets}
 {\em Any didendritic system with $\# \mathcal N \ge 3$ has  the following property}
\begin{itemize}
\item[(F)] For distinct $i,j,k \in \mathcal N$, there exists a bijective mapping from $\{\langle g,h\rangle: g,h \in \{i,j,k\}\}$ to the set of vertices of a full binary tree ${\tilde{\mathbf  t}}$ with three leaves labeled by $i,j,k$ which  preserves the three partial orders and, for $g \in \{i,j,k\}$, maps $\langle g,g\rangle$ to the leaf of ${\tilde{\mathbf  t}}$ labeled $g$.
\end{itemize}
\end{lemma}
\begin{proof} Assume that the first of the three possible set of relations in (C) is valid, i.e.
  $\langle i,j\rangle = \langle i,k\rangle  < \langle j,k\rangle$. Then we must have by (B) that 
$\langle i,j\rangle = \langle i,k\rangle  < \langle j,k\rangle < j$, 
and 
$\langle i,j\rangle < i$.

Applying (B) to $i$ and $j$, we have either $\langle i,j\rangle <_L i$ and $\langle i,j\rangle <_R j$ or $\langle i,j\rangle <_L j$ and $\langle i,j\rangle <_R i$. Assume the former. Then (E) (together with (B)) enforces 
$\langle i,j\rangle <_R \langle j,k\rangle$.

Now applying  (B) to $j$ and $k$, we have either $\langle j,k\rangle <_L j$ and $\langle j,k\rangle <_R k$ or $\langle j,k\rangle <_L k$ and $\langle j,k \rangle <_R j$. Assume the former.  This then results in the full binary tree $\tilde{\mathbf t}$ with vertex set $\{\emptyset, 0,1, 10,11\}$, where  $\emptyset$ is labeled by $\langle i,j\rangle = \langle i,k\rangle$, $1$ is labeled by $\langle j,k\rangle$, and the three leaves $0$, $10$ and $11$ are labeled by $i$, $j$ and $k$, respectively.

Combining all the ways allowed by (C) and (B), we arrive at the $3\times 2\times 2 = 12$~possible leaf-labeled full binary trees with three leaves and label set $\{i,j,k\}$. 
\end{proof}

Let $\tilde{\mathbf t}$ 
be a leaf-labeled full binary tree,  $\mathcal D$ be a didendritic system with finite label set $\mathcal N$, 
and $\psi$ be a mapping from the set of equivalence classes  of $\mathcal D$ to the set of vertices of  $\tilde{\mathbf t}$. 
We call $\psi$ an {\em isomorphism} from $\mathcal D$ to $\tilde{\mathbf t}$ if it is bijective and preserves the three partial orders. 
Necessarily, $\psi$ then maps the label set $\mathcal N$ bijectively to the set of leaves of $\tilde{\mathbf t}$.

\begin{proposition} 
\label{DDS_tree}
Let $\mathcal D = (\mathcal N, \, \equiv, \, \langle\cdot, \cdot \rangle, \, <_L, \, <_R, \, <)$ be a finite didendritic system. Then $\mathcal D$ is isomorphic to some leaf-labeled full binary tree 
$\tilde{\mathbf t}$.  Moreover, both $\tilde{\mathbf t}$ and the isomorphism are unique.
\end{proposition}

\begin{proof} Denote the cardinality of $\mathcal{N}$ by $n$.

For   $n=2$ the assertion follows directly from properties (A)-(C). 

For $n\ge 3$ we will use that $\mathcal D$ also has property (F) because of Lemma~\ref{ordered_triplets}. 
Indeed, for $n=3$ the assertion is immediate from that property.  

 We now proceed by induction and assume that the assertion is true for some $n \ge 3$.  

For the induction step, assume without loss of generality that 
$\mathcal N_{n+1} =[n+1]$, and let $\mathcal D_{n+1}$ obey the assumptions of the proposition with label set $[n+1]$.
 We are going to construct a unique leaf-labeled full binary tree $\tilde{\mathbf t}_{n+1}$  and isomorphism $\psi$ from $\mathcal D_{n+1}$ to $\tilde{\mathbf t}_{n+1}$.

To this end we denote the restriction of $\mathcal D_{n+1}$ to $[n]$ by $\mathcal D_n$. Since $\mathcal D_n$ meets the conditions of the induction hypothesis, there exists a unique leaf-labeled full binary tree $\tilde{\mathbf t}_n$ and isomorphism $\phi$ from $\mathcal D_n$ to $\tilde{\mathbf t}_n$. By a slight abuse of notation we will use the  symbol $\langle i,j\rangle$ both for the equivalence classes in $\mathcal D_n$ and for the equivalence classes in $\mathcal D_{n+1}$.

Now choose $a, b \in [n]$ such that $\phi(\langle a,b\rangle) = \emptyset \in  \tilde{\mathbf t}_n$.
Focusing on $a, b, n+1$, the triplet property (C) allows to distinguish between two cases.
The first case is when $\langle n+1, a\rangle = \langle n+1$, $b \rangle  < \langle a, b \rangle$
and the second case is when $\langle a, b \rangle = \langle n+1$, $b \rangle < \langle n+1, a \rangle$ or $\langle a$, $b \rangle = \langle n+1, a \rangle < \langle n+1, b  \rangle$

\bigskip
\noindent
{\bf Case 1:} $\langle n+1, a\rangle = \langle n+1, b \rangle  < \langle a, b \rangle$. \\
 Arguing by the isomorphy of $\mathcal D_n$ and $\tilde{\mathbf t}_n$, and recalling that $\phi(\langle a,b\rangle) = \emptyset$, we see that for all $h\in [n]$ we have that either 
$\langle a, b \rangle = \langle a, h \rangle < \langle a, h\rangle$ 
or 
$\langle a, b \rangle = \langle a, h \rangle < \langle b, h\rangle$. 
In the former case,  we obtain 
from 
$\langle n+1, a\rangle = \langle n+1$, $b \rangle  < \langle a, b \rangle$,
$\langle a, b \rangle = \langle a, h \rangle < \langle a, h\rangle$,
and the transitivity of $<$ that 
$\langle n+1, a\rangle < \langle a, h\rangle$
and thence,
from the triplet property (C) applied to the triplet $a, h, n+1$,
that $\langle n+1, h\rangle = \langle n+1, a\rangle < \langle a, h \rangle$. 
In the latter case, we obtain by a similar argument that $\langle n+1, h\rangle = \langle n+1, b \rangle < \langle b, h \rangle$. 
Thus, in both cases we have for all $h \in [n]$ that $\langle n+1, h\rangle = \langle n+1, a\rangle=\langle n+1, b \rangle $. 
Since for all $i,j \in [n]$ with $\langle i,j\rangle \neq \langle a,b\rangle$ we have $\langle a,b\rangle < \langle i,j\rangle$ by the induction assumption, the transitivity of $<$ implies that 
$\langle n+1,h\rangle = \langle n+1,a\rangle = \langle n+1,b\rangle< \langle i,j\rangle$ for all $i,j \in [n]$. 
Thus, we are led to define 
\[ \psi(\langle n+1, h \rangle) := \emptyset \; \text{for} \; h \in [n].
\] 

To define the other values of $\psi$, we consider two subcases:

\medskip
\noindent
{\bf Case 1a:} $\langle n+1$, $a\rangle = \langle n+1, b \rangle <_L \langle a, b \rangle$. \\
Recall from above that $\langle a,b \rangle \le \langle i,j \rangle$ for all $i,j \in [n]$.
Property (E) then ensures that $\langle n+1, a\rangle = \langle n+1,b \rangle <_L \langle i, j \rangle$ for all $i,j \in [n]$. 
Moreover, the triplet property (C) applied to the triplet $n+1, a,b$ combined with property (D) guarantees that $\langle n+1, a\rangle = \langle n+1, b\rangle <_R \langle n+1, n+1 \rangle$. 
Consequently, we are led to set
\begin{align*}
\psi(\langle n+1, n+1 \rangle) &:= 0,\\
\psi(\langle i, j \rangle) &:= 1\phi(\langle i, j \rangle)\mbox{ for } i,j\in [n].
 \end{align*}

\medskip
\noindent
 {\bf Case 1b:} $\langle n+1, a\rangle = \langle n+1, b \rangle <_R \langle a, b \rangle$. \\ 
By an argument similar to that in Case 1b, we are led to set 
\begin{align*}
\psi(\langle n+1, n+1 \rangle) &:= 1,\\
\psi(\langle i, j \rangle) &:= 0\phi(\langle i, j \rangle)\mbox{ for } i,j\in [n].
 \end{align*}

\bigskip
\noindent
 {\bf Case 2:} $\langle a, b \rangle = \langle n+1, b \rangle < \langle n+1, a \rangle$ or $\langle a, b \rangle = \langle n+1, a \rangle < \langle n+1, b  \rangle$. \\
From property (F) for $a,b, n+1$, we see that either of the two above subcases implies that $\langle a, b\rangle < n+1$. Also, we may assume without loss of generality from property (B) 
(by exchanging the roles of $a$ and $b$ if necessary) that $\langle a, b \rangle <_L a$ and $\langle a, b \rangle <_R b$.

Put $\mathcal L := \{g\in [n+1] : \langle a,b \rangle <_L g\}$, $\mathcal R := \{h\in [n+1] : \langle a,b \rangle <_R h\}$, and note 
from property (D) that $n+1 \in \mathcal L \cup  \mathcal R$. Moreover, $a \in \mathcal L$ and $b \in \mathcal R$, hence
$\# \mathcal L \le n$ and $\# \mathcal R \le n$.  We may thus apply the induction assumption to $\mathcal D_{\mathcal L}$ and $\mathcal D_{\mathcal R}$,   defined to be the restrictions of $\mathcal D_{n+1}$ to~$ \mathcal L$ and $ \mathcal R$, respectively. Let $\tilde{\mathbf t}_{\mathcal L}$ and $\tilde{\mathbf t}_{\mathcal R}$ be the corresponding leaf-labeled full binary trees, and let $\chi_{\mathcal L}$ and $\chi_{\mathcal R}$ be the corresponding isomorphisms.  
The desired leaf-labeled full binary tree 
$ \tilde {\mathbf t}_{n+1}$
 has the vertex set $\{\emptyset\} \cup \{0v: v \in \tilde{\mathbf t}_{\mathcal D}\} \cup \{1v: v\in \tilde {\mathbf t}_{\mathcal R}\}$, and the isomorphism $\psi$  from $\mathcal D_{n+1}$ to $\tilde{\mathbf t}_{n+1}$ is given by
\begin{itemize}
\item
$\psi(\langle g, h\rangle) := \emptyset $\quad \quad  for $g\in \mathcal L$, $h\in \mathcal R$,
\item
$\psi(\langle g, h\rangle) := 0v$
  \quad for $g, h\in \mathcal L$, with $v:=\chi_{\mathcal L}(\langle g, h\rangle)$
\item
 $\psi(\langle g, h\rangle) := 1v$
 \quad  for $g, h\in \mathcal R$, with $v:=\chi_{\mathcal R}(\langle g, h\rangle)$.
\end{itemize}
 \end{proof}
%
%
%
The next result shows that a finite didendritic system can be constructed in two stages. The first stage determines the partial order $<$ while the second stage resolves each instance of $<$ as either $<_L$ or $<_R$ in a consistent manner.  A randomized version of this construction appears in the statement of the main representation result, Theorem 8.2 of \cite{MR3601650}. In the next proposition we elaborate on the ``deterministic heart'' of that construction.

\begin{proposition}
\label{P:left_right}
 Let $\mathcal N$ be a finite set, $\equiv$ be an equivalence relation on $\mathcal N \times \mathcal N$ and $<$ be a partial order on the set of equivalence classes. Suppose also that for distinct $i,j \in\mathcal N$ there are 
elements $w(i,j)$ of the set $\{\curvearrowright, \curvearrowleft\}$. Assume properties  (A) and (C) from Definition~\ref{D:didendritic}, as well as the following properties
\begin{itemize}
\item[(B')] For distinct $i,j \in \mathcal N$, $\langle i,j \rangle < i$ and $\langle i,j \rangle < j$.
\item[(B'')]
For $i \ne j$,
$w(i,j) = \curvearrowright$ if and only if $w(j,i) = \curvearrowleft$.
\item[(E')]
For distinct $i,j,k \in \mathcal N$,
if $\langle i,j\rangle =\langle i,k \rangle < \langle j,k\rangle$,
then $w(i,j) = w(i,k)$.
\end{itemize}
Then there is a unique pair of partial orders $<_L$ and $<_R$ on $\{\langle i,j \rangle : i,j \in\mathcal N\}$
such that  
\[
\langle i,j \rangle <_L i \text{ and } \langle i,j \rangle <_R j
\; \Longleftrightarrow \;
w(i,j) = \curvearrowright
\]
and the ensemble $(\mathcal N, \, \equiv, \, \langle \cdot, \cdot \rangle, \, <_L, \, <_R, \, <)$
is a didendritic system.
\end{proposition}
\begin{proof} We may assume without loss of generality that $\mathcal N = [n]$ for some $n \in \mathbb N$.

\bigskip
\noindent
Step 1. For $n=1$ there is nothing to prove.

\bigskip
\noindent
Step 2. For $n=2$ we have two cases according to property (B'): \\ If $w(1,2) = \curvearrowright$ then $\langle 1,2\rangle <_L 1$ and  $\langle 1,2\rangle <_R 2$. If $w(1,2) = \curvearrowleft$ then $\langle 1,2\rangle <_L 2$ and  $\langle 1,2\rangle <_R 1$. In both cases, $(\{1,2\}, \equiv, \langle \cdot, \cdot \rangle, <_L, <_R, <)$
is a didendritic system. 

\bigskip
\noindent
Step 3. Now assume  $n\ge3$. 
 
 To show property (B) in the definition of a didendritic system we can argue as in the case $n=2$. 

Next we show for distinct $i,j,k \in [n]$ that the assumptions of the proposition define partial orders $<_L, <_R$ on 
$\{\langle g,h\rangle: g,h \in \{i,j,k\}\}$ which meet condition~(F) formulated in Lemma~\ref{ordered_triplets}.

\bigskip
Indeed, because of condition (C) we have one of the following cases: 

\medskip
\noindent
Case 1: \qquad $\langle i,j\rangle = \langle i,k\rangle < \langle j,k\rangle$,
\\
Case 2: \qquad $\langle j,k\rangle = \langle j,i\rangle < \langle k,i \rangle$,
\\
Case 3: \qquad $\langle k,i\rangle = \langle k,j\rangle < \langle i,j \rangle$.

\bigskip
Because of property (E'), in Case 1 we have $w(i,j) = w(i,k)$, in Case 2 we have $w(j,k) = w(j,i)$, and in Case 3 we have $w(k,i) = w(k,j)$. 
Thus, in the three cases, we have the four different choices  $(\curvearrowright,\curvearrowright), (\curvearrowright,\curvearrowleft), (\curvearrowleft, \curvearrowright), (\curvearrowleft, \curvearrowleft)$ for the pair $(w(i,j), w(j,k))$ in Case~1, 
for the pair $(w(j,i), w(k,i))$ in Case 2, and for the pair $(w(k,i), w(i,j)$ in Case 3. Any of these $3\times4=12$ sub-cases 1a), \ldots, 3d) 
leads to partial orders $<_L, <_R$ on $\{\langle g,h\rangle: g,h \in \{i,j,k\}\}$ and to a (distinct) leaf-labeled full binary tree ${\tilde{\mathbf  t}}$ 
whose five vertices are bijectively labeled by the elements of $\{\langle g,h\rangle: g,h \in \{i,j,k\}\}$ in an order preserving way.
 For example, in  Case 1a) ${\tilde{\mathbf  t}}$ consists of the vertex set $\{\emptyset, 0, 1, 10, 11\}$, with $0$ labeled $i$, $10$ labeled $j$ and $11$ labeled $k$.
In Case 3c), ${\tilde{\mathbf  t}}$ consists of the vertex set $\{\emptyset, 0, 00, 01, 1\}$, with $00$ labeled $i$, $01$ labeled $j$ and $1$ labeled $k$.

\bigskip
\noindent
Step 4. We now proceed inductively to show also the remaining properties (D) and (E) in the definition of a didendritic system. 

Let $\mathcal N_{n+1}=[n+1]$ and consider the restrictions of $\equiv$ and $<$ to $[n]\times [n]$. By the induction hypothesis there is a unique didendritic system $\mathcal D_n$ with label set $[n]$ that has the properties stated in the proposition. By Proposition~\ref{DDS_tree} there is a unique leaf-labeled full binary tree  $\tilde{\mathbf t}_n$ whose $n$ vertices are bijectively labeled by the elements of $[n]$ such that the partial orders $<_L$, $<_R$ and $<$ are preserved. Again by Proposition~\ref{DDS_tree}, for the induction step it suffices to construct out of $\tilde{\mathbf t}_n$ a leaf-labeled full binary tree  $\tilde{\mathbf t}_{n+1}$ whose $2n+1$ vertices are bijectively labeled by the equivalence classes $\langle i, j\rangle$, $i,j\in [n+1]$, such that the partial orders $<_L$, $<_R$ and $<$ are preserved, and to show that the construction of  $\tilde{\mathbf t}_{n+1}$  is unique.

Let $a,b \in [n]$ be such that the root $\emptyset$ of $\tilde{\mathbf t}_n$ is labeled by $\langle a, b \rangle$. 
Because of condition~(C) applied to the triplet $n+1, a,b$, one of the following three cases applies:

\medskip
\noindent
Case 1: \qquad \, $\langle n+1,a\rangle = \langle n+1, b\rangle < \langle a,b\rangle$
\\
Case 2a: \qquad $\langle a,b\rangle = \langle a,n+1\rangle < \langle b,n+1 \rangle$
\\
Case 2b: \qquad $\langle b,n+1\rangle = \langle b,a\rangle < \langle n+1,a \rangle$

\medskip

In Case 1, because of the transitivity of $<$ and the isomorphy between $\mathcal D_n$ and $\tilde{\mathbf t}_n$, we have $\langle n+1, a\rangle  = \langle n+1, b\rangle < \langle i,j\rangle$ for all $i,j \in [n]$. Therefore, within $\tilde{\mathbf t}_{n+1}$, the root $\emptyset$ must be labeled by $\langle n+1, a\rangle = \langle n+1, b\rangle$, and either the vertex $1$ or the vertex $0$ must be a leaf labeled by $n+1$. The former is the case if $w(n+1,a) = \curvearrowleft$, the latter if $w(n+1,a) = \curvearrowright$. In the former case we have   $\langle n+1, a\rangle = \langle n+1, b\rangle <_L \langle a, b \rangle$, and the addresses in $\tilde{\mathbf t}_{n+1}$ of the equivalence classes $\langle i, j\rangle$, $i,j \in [n+1]$,  are then given just as in Case 1a in the proof of Proposition~\ref{DDS_tree}. In the latter case we have $\langle n+1, a\rangle = \langle n+1, b\rangle <_R \langle a, b \rangle$, and the addresses in $\tilde{\mathbf t}_{n+1}$ of the equivalence classes $\langle i, j\rangle$, $i,j \in [n+1]$,  are then given just as in Case 1b in the proof of Proposition~\ref{DDS_tree}.

\medskip

Now we assume Case 2a and we will parallel the reasoning in Case 2 in the proof of Proposition~\ref{DDS_tree}. From Step 3 we have the property (F) in hand; this applied to the triplet $\{i,j,k\} = \{a,b,n+1\}$ ensures that the set of equivalence classes $\{\langle g,h\rangle: g,h \in \{a,b,n+1\}$ is isomorphic to a labeled binary tree $\tilde{\mathbf t}$ with 3 leaves, for which either $0$ or $1$ is a leaf. Assuming without loss of generality that $w(a,b) = \curvearrowright$ we see that the vertex $0$ in $\tilde{\mathbf t}$ must be labeled with $a$ and the  vertex $1$ in $\tilde{\mathbf t}$ must be labeled with $\langle b,n+1\rangle$. Thus, we arrive that  $\langle a,b\rangle = \langle a,n+1\rangle <_R \langle b,n+1\rangle$ and  $\langle a,b\rangle <_L a$.

Put $\mathcal L := \{h\in [n] : \langle a,b \rangle <_L h\}$ and $\mathcal R := \{h\in [n] : \langle a,b \rangle <_R h\} \cup\{n+1\}$.  Note that $\# \mathcal L \le n$ and $\# \mathcal R \le n$ because $a \in \mathcal L$ and $b \in \mathcal R$.  We may thus apply the induction assumption to the restrictions of $\equiv$, $<$ and $w$ to the two set $ \mathcal L$ and $ \mathcal R$, respectively. 
Let $\tilde{\mathbf t}_{\mathcal L}$ and $\tilde{\mathbf t}_{\mathcal R}$ be the corresponding leaf-labeled full binary trees, and let $\chi_{\mathcal L}$ and $\chi_{\mathcal R}$ be the corresponding isomorphisms.  
The desired leaf-labeled full binary tree 
$\tilde{\mathbf t}_{n+1}$
 has the vertex set $\{\emptyset\} \cup \{0v: v \in \tilde{\mathbf t}_{\mathcal L}\} \cup \{1v: v\in \tilde{\mathbf t}_{\mathcal R}\}$, 
and the addresses of the equivalence classes $\langle g,h\rangle$, $g,h\in [n+1]$, in the tree $\tilde{\mathbf t}_{n+1}$ are given by
\begin{itemize}
\item
$\psi(\langle g, h\rangle) := \emptyset $\quad  for $g\in \mathcal L$, $h\in \mathcal R$,
\item
$\psi(\langle g, h\rangle) := 0v$
  \quad for $g, h\in \mathcal L$, with $v:=\chi_{\mathcal L}(\langle g, h\rangle)$,
\item
 $\psi(\langle g, h\rangle) := 1v$
 \quad  for $g, h\in \mathcal R$, with $v:=\chi_{\mathcal R}(\langle g, h\rangle)$.
 \end{itemize}

 The partial orders $<_L$ and $<_R$ are then simply read off from the corresponding orders in the leaf-labeled full binary tree $\tilde{\mathbf t}_{n+1}$, and clearly also the properties (D) and (E) are inherited.

\medskip

 It remains to deal with the case 2b. This is, however, completely analogous to case 2a, and we refrain from giving the parallel arguments here. 
\end{proof}

We have seen that finite didendritic systems are essentially finite leaf-labeled full binary trees. One way to produce didendritic systems with infinite label sets is via a projective limit construction as detailed in the following two lemmas. 
We omit the (simple) proofs. The first lemma says that the natural ``projection'' of a didendritic system to a subset of its label set is again a didendritic system.

\begin{lemma}
\label{L:restriction}
Consider a didendritic system $(\mathcal{N}, \, \equiv, \, \langle \cdot, \cdot \rangle, \, <_L, \, <_R, \, <)$.
Let $\mathcal{N}'$ be a nonempty subset of $\mathcal{N}$.  Define an equivalence relation
$\equiv'$ on $\mathcal{N}' \times \mathcal{N}'$ by declaring that $(i,j) \equiv' (k,\ell)$
if and only if $(i,j) \equiv (k,\ell)$.  Write $\langle i, j \rangle'$ for the equivalence
class of $(i,j) \in \mathcal{N}' \times \mathcal{N}'$.  Define a partial order $<_L'$ on the
equivalence classes of $\equiv'$ by declaring that
$\langle i,j \rangle' <_L' \langle k,\ell \rangle'$ if and only if
$\langle i,j \rangle <_L \langle k,\ell \rangle$.  Define partial orders
$<_R'$ and $<'$ analogously.  Then $(\mathcal{N}', \, \equiv', \, \langle \cdot, \cdot \rangle', \, <_L', \, <_R', \, <')$
is a didendritic system.
\end{lemma}

\begin{definition}
\label{D:restriction}
Suppose that the didendritic systems $(\mathcal{N}, \, \equiv, \, \langle \cdot, \cdot \rangle, \,  <_L, \,  <_R, \, <)$ and
$(\mathcal{N}', \, \equiv', \, \langle \cdot, \cdot \rangle', \, <_L', \,  <_R', \, <')$ are as in Lemma~\ref{L:restriction}.
We then say that the latter didendritic system is the {\em restriction} of the former to $\mathcal{N}'$.
\end{definition}

The following lemma asserts the existence of a projective limit for a projective family of didendritic systems.

\begin{lemma}
\label{L:projective}
Suppose that $(\mathcal{N}^n, \equiv^n, \langle \cdot, \cdot \rangle^n, <_L^n, <_R^n, <^n)$, $n \in \mathbb{N}$,
are didendritic systems such that for $m < n$ we have $\mathcal{N}^m \subseteq \mathcal{N}^n$ and
$(\mathcal{N}^m, \equiv^m, \langle \cdot, \cdot \rangle^m, <_L^m, <_R^m, <^m)$
is the restriction of $(\mathcal{N}^n, \equiv^n, \langle \cdot, \cdot \rangle^n, <_L^n, <_R^n, <^n)$
to $\mathcal{N}^m$.  Put $\mathcal{N}^\infty := \bigcup_{n \in \mathbb{N}} \mathcal{N}^n$.  Then there is a unique
didendritic system 
$(\mathcal{N}^\infty, \equiv^\infty, \langle \cdot, \cdot \rangle^\infty, <_L^\infty, <_R^\infty, <^\infty)$
such that for each $n \in \mathbb{N}$ the didendritic system 
$(\mathcal{N}^n, \equiv^n, \langle \cdot, \cdot \rangle^n, <_L^n, <_R^n, <^n)$
is the restriction of 
$(\mathcal{N}^\infty, \equiv^\infty, \langle \cdot, \cdot \rangle^\infty, <_L^\infty, <_R^\infty, <^\infty)$ to
$\mathcal{N}^n$.
\end{lemma}

\section{Infinite PATRICIA bridges and \\ exchangeable random didendritic systems}
\label{S:bridges_DDS}

It is shown in Section 5 of \cite{MR3601650} that if $(\bar R_n^\infty)_{n \in \mathbb{N}}$
is an infinite PATRICIA bridge, then there is a Markov chain
$(\tilde R_n^\infty)_{n \in \mathbb{N}}$
such that for each $n \in \mathbb{N}$ 
the random element $\tilde R_n^\infty$ is a leaf-labeled full binary
tree with $n$ leaves labeled by $[n]$
and the following hold.
\begin{itemize}
\item
The full binary tree obtained by removing the labels of $\tilde R_n^\infty$
is $\bar R_n^\infty$.
\item
For every $n \in \mathbb{N}$, the conditional distribution
of $\tilde R_n^\infty$ given $\bar R_n^\infty$ is uniform over the
$n!$ possible leaf-labelings of $\bar R_n^\infty$ by [n].
\item
In going backward from time $n+1$ to time $n$,
$\tilde R_{n+1}^\infty$ is transformed into $\tilde R_n^\infty$
according to the following deterministic procedure:
\begin{itemize}
\item
Delete the leaf labeled $n+1$, along with its sibling
(which may or may not be a leaf).
\item 
If the sibling of the leaf labeled $n+1$ is also a leaf, then 
assign the sibling's label to the common parent (which is now a leaf).
\item
If the sibling of the leaf labeled $n+1$ is not a leaf, then attach
the subtree below the sibling to the common parent with its 
leaf labels unchanged and leave all other leaf labels unchanged.
\end{itemize}
\end{itemize}

As we saw in Definition~\ref{D:finite_didendritic}, the leaf-labeled full binary tree $\tilde R_n^\infty$ defines
a random didendritic system with label set $[n]$ for any $n \in \mathbb{N}$.  Moreover, for any $n \in \mathbb{N}$ the didendritic system
defined by $\tilde R_n^\infty$ is the restriction to $[n]$ of the didendritic system defined by $\tilde R_{n+1}^\infty$.
It follows from Lemma~\ref{L:projective} that there is a random didendritic system $(\mathbb{N}, \, \equiv, \, \langle \cdot, \cdot \rangle, \, <_L, \, <_R, \, <)$
such that the restriction of this random didendritic system to $[n]$ is the random didendritic system defined by $\tilde R_n^\infty$ for all $n \in \mathbb{N}$.
Because of Proposition~\ref{DDS_tree} we can recover $\tilde R_n^\infty$ and hence $\bar R_n^\infty$ from this restriction and therefore we can recover
$(\tilde R_n^\infty)_{n \in \mathbb{N}}$ and $(\bar R_n^\infty)_{n \in \mathbb{N}}$ from
$(\mathbb{N}, \, \equiv, \, \langle \cdot, \cdot \rangle, \, <_L, \, <_R, \, <)$.  

The random didendritic system defined by $(\tilde R_n^\infty)_{n \in \mathbb{N}}$ is not arbitrary: it inherits distributional symmetries from
the uniform labeling in the construction of $(\tilde R_n^\infty)_{n \in \mathbb{N}}$ from $(\bar R_n^\infty)_{n \in \mathbb{N}}$.  We now develop
some terminology to describe these symmetries.

\begin{definition}\label{def5_8}
Given a didendritic system 
$\mathcal{D} = (\mathbb{N}, \, \equiv, \, \langle \cdot, \cdot \rangle, \, <_L, \, <_R, \, <)$ with label set $\mathbb{N}$
and a permutation $\sigma$ of $\mathbb{N}$ such
that $\sigma(i) = i$ for all but finitely many $i \in \mathbb{N}$,
the didendritic system 
$\mathcal{D}^\sigma = (\mathbb{N}, \equiv^\sigma, \, \langle \cdot, \cdot \rangle^\sigma, \, <_L^\sigma, \, <_R^\sigma, <^\sigma)$ is defined by
\begin{itemize}
\item
$(i',j') \equiv^\sigma (i'',j'')$ 
if and only if 
$(\sigma(i'), \sigma(j')) \equiv (\sigma(i''), \sigma(j''))$,
\item
$\langle i, j \rangle^\sigma$ is the equivalence class of
the pair $(i,j)$ for the equivalence relation $\equiv^\sigma$,
\item
$\langle h,i \rangle^\sigma <_L^\sigma \langle j,k \rangle^\sigma$
if and only if
$\langle \sigma(h), \sigma(i) \rangle <_L \langle \sigma(j), \sigma(k) \rangle$,
\item
$\langle h,i \rangle^\sigma <_R^\sigma \langle j,k \rangle^\sigma$
if and only if
$\langle \sigma(h), \sigma(i) \rangle <_R \langle \sigma(j), \sigma(k) \rangle$,
\item
\item
$\langle h,i \rangle^\sigma <^\sigma \langle j,k \rangle^\sigma$
if and only if
$\langle \sigma(h), \sigma(i) \rangle <R \langle \sigma(j), \sigma(k) \rangle$.
\end{itemize}
A random didendritic system 
$\mathcal{D} = (\mathbb{N}, \, \equiv, \, \langle \cdot, \cdot \rangle, \, <_L, \, <_R, \, <)$
with label set $\mathbb{N}$  
is {\em exchangeable} if for each permutation $\sigma$ of $\mathbb{N}$ such
that $\sigma(i) = i$ for all but finitely many $i \in \mathbb{N}$ 
the random didendritic system 
$\mathcal{D}^\sigma$
has the same distribution as
$\mathcal{D}$.
\end{definition}

The following result is Lemma~5.12 of \cite{MR3601650}.

\begin{lemma}
The random didendritic system on $\mathbb{N}$ corresponding to the labeled version
of an infinite PATRICIA bridge is exchangeable.  Conversely, if we have an exchangeable
random didendritic system (on $\mathbb{N}$), apply Lemma~\ref{L:restriction} to restrict it to
a sequence of random didendritic systems on $[n]$ for $n \in \mathbb{N}$, and apply Proposition~\ref{DDS_tree}
to construct a full binary tree that is leaf-labeled by $[n]$, then the resulting sequence of leaf-labeled
full binary trees is a PATRICIA bridge.
\end{lemma}

Our aim is to find concrete representations of the extremal infinite R\'emy bridges 
(recall that an infinite R\'emy bridge is extremal if it has a trivial tail $\sigma$-field). 
To this end, it will be useful to relate the extremality of an infinite R\'emy bridge to 
properties of the associated exchangeable random didendritic system.
We say that an exchangeable random didendritic system $\mathcal{D}$ is {\em ergodic} if 
\[
\mathbb{P}(\{\mathcal{D} \in A\} \triangle \{\mathcal{D}^\sigma \in A\}) = 0
\] 
for all permutations $\sigma$ of $\mathbb{N}$ such
that $\sigma(i) = i$ for all but finitely many $i \in \mathbb{N}$ implies that 
\[
\mathbb{P}\{\mathcal{D} \in A\} \in \{0,1\}.
\] 
The following result is Proposition~5.19 of \cite{MR3601650}.

\begin{proposition}
\label{P:extremal_vs_ergodic}
An infinite R\'emy bridge is extremal if and only if
the associated exchangeable random didendritic system is ergodic.
\end{proposition}
\begin{example}\label{E:zigzagDDS}
We return to Example~\ref{E:zig-zag_path} and give a concrete representation of the corresponding random didendritic system. Let $U_1, U_2,\ldots$ be independent and uniformly distributed on $[0,1]$, and let $\varepsilon_1, \varepsilon_2, \ldots$ be i.i.d. with values in $\{\curvearrowright, \curvearrowleft\}$ and with $\mathbb{P}\{\varepsilon_i=\curvearrowright\}=\mathbb{P}\{\varepsilon_i=\curvearrowleft\}=\frac12$. We define the equivalence relation $\equiv$ on $\mathbb{N}\times \mathbb{N}$ by declaring that 
\[(i,j) \equiv (k,k) \Longleftrightarrow i=j=k\]
and if $i\neq j$ and $k\neq \ell$ then
\[(i,j)\equiv(k,\ell) \Longleftrightarrow U_i\wedge U_j = U_k\wedge U_\ell.\]
The partial order $<$ is defined by declaring that 
\begin{itemize}
\item for $i,j,k \in \mathbb N$, \, $\langle i,j\rangle < \langle k,k\rangle \Longleftrightarrow i\neq j \; \mbox{ and }\; U_i\wedge U_j \le U_k$,
\item  for $i\neq j$, $k\neq \ell$, \, 
$\langle i,j\rangle < \langle k, \ell\rangle  \Longleftrightarrow  i\neq j \; \mbox{ and } \; U_i\wedge U_j < U_k\wedge U_\ell.$
\end{itemize}
We say that $\langle i,j\rangle <_L \langle k, \ell\rangle$ (respectively, $\langle i,j\rangle <_R \langle k, \ell\rangle$) if and only if $\langle i,j\rangle < \langle k, \ell\rangle$ and $\varepsilon_h = \curvearrowleft$ (respectively, $\varepsilon_h = \curvearrowright$), where 
\[ h = \begin{cases} i, \mbox { if }  U_i < U_j, \\ j, \mbox { if }  U_j < U_i. \end{cases}
\]
One can check that the restriction of this random didendritic system to the label set $[n]$ defines a random full binary tree which has the same distribution as the random full binary tree $\bar R_n^\infty$ in \ref{E:zig-zag_path}. The labeling of this tree by $[n]$ is clearly uniform, and the result of passing from the restriction to $[n+1]$ to the restriction to $[n]$ is given by the deterministic transformation described at the beginning of this section. Together this shows that the corresponding full binary tree-valued process is distributed as the labeled version of the PATRICIA bridge from Example~\ref{E:zig-zag_path}.

It is clear by the symmetry inherent in the construction that the above random didendritic system is exchangeable.
We already saw in Example  \ref{E:zig-zag_path} that the above PATRICIA bridge is extremal, and hence by Proposition~\ref{P:extremal_vs_ergodic} the exchangeable random didendritic system is ergodic. The ergodicity can also be seen directly from the observation that the restrictions of the random didendritic system to disjoint finite subsets of $\mathbb{N}$ are independent 
(see Remark~5.18 of \cite{MR3601650}).
\end{example}

\section{Infinite PATRICIA bridges and real trees}
\label{S:DDS_from_real_tree}
The first part of Theorem 8.2 of \cite{MR3601650} constructs an ergodic exchangeable random didendritic system  $\mathcal D$ 
(and hence an extremal infinite PATRICIA bridge) from an \mbox{$\mathbb{R}$-tree}~$\mathbf{S}$ with root $\rho$, a probability measure $\mu$ on $\mathbf{S}$, 
and a function $W: \mathbf{S} \times [0,1] \times \mathbf{S} \times [0,1] \to \{\curvearrowright, \curvearrowleft\}$. For the purposes of this construction, we first fix some notation. 
For $x, y \in \mathbf{S}$ let $x\curlywedge y$ denote that element  in the segment $[x,y]$ that is closest to the root $\rho$ of $\mathbf{S}$
(equivalently, $[\rho, x \curlywedge y] = [\rho, x] \cap [\rho,y]$).
 We say that $x \prec y$ if $x \in [\rho, y)$. 

The following properties of $\mathbf{S}$, $\rho$, $\mu$, and $W$ are essential:

\medskip
\noindent
{\em (T) Let $(\xi_n)_{n \in \mathbb{N}}$ be i.i.d. with common distribution $\mu$.  Then 
almost surely for distinct $i,j,k \in \mathbb{N}$, one of 
\[
\xi_i \curlywedge \xi_j = \xi_i \curlywedge \xi_k \prec \xi_j \curlywedge \xi_k,
\]
\[
\xi_j \curlywedge \xi_k = \xi_j \curlywedge \xi_i \prec \xi_k \curlywedge \xi_i,
\]
or
\[
\xi_k \curlywedge \xi_i = \xi_k \curlywedge \xi_j \prec \xi_i \curlywedge \xi_j
\]
holds. }

\medskip
\noindent
{\em (LR) For 
an independent sequence of i.i.d. $[0,1]$-valued random variables $(\vartheta_n)_{n \in \mathbb{N}}$ with common uniform distribution one has almost surely
\begin{itemize}
\item
for $i \ne j$, 
$W(\xi_i, \vartheta_i, \xi_j, \vartheta_j) = \curvearrowright$
if and only if
$W(\xi_j, \vartheta_j, \xi_i, \vartheta_i) = \curvearrowleft$;
\item
for distinct $i,j,k$,
if $\xi_i \curlywedge \xi_j = \xi_i \curlywedge \xi_k \prec \xi_j \curlywedge \xi_k$,
then $W(\xi_i ,\vartheta_i, \xi_j, \vartheta_j) = W(\xi_i ,\vartheta_i, \xi_k, \vartheta_k)$.
\end{itemize}}

\medskip

Assume (T) and (LR) hold.
Let $(\xi_i, \vartheta_i)$ be i.i.d. copies of a random variable with distribution $\mu\otimes\lambda$, where $\lambda$ is the uniform distribution on $[0,1]$.  Using the random input  $(\xi_i, \vartheta_i)_{i\in \mathbb N}$, we define 
\begin{itemize}
\item
the equivalence relation $\equiv^{\mathbf{S}}$ on $\mathbb N\times \mathbb N$ by declaring \\ for $i,k,\ell \in \mathbb N$ that $(i,i)\equiv^\mathbf{S}(k,\ell)$ if and only if $i=k=\ell$, \\ and for $i\neq j$, $k\neq \ell$, that $(i,j)\equiv^\mathbf{S}(k,\ell)$ if and only if $\xi_i \curlywedge \xi_j = \xi_k\curlywedge \xi_\ell$,
\item
the partial order $<^{\mathbf{S}}$ on the equivalence classes $\langle \cdot,\cdot\rangle^\mathbf{S}$ of $\equiv^{\mathbf{S}}$ by declaring\\
for $i,k,\ell \in \mathbb N$ that $\langle i,j\rangle^\mathbf{S} <^{\mathbf{S}} \langle k,k\rangle^\mathbf{S}$ if and only if $i \neq j$ and $\xi \curlywedge \xi_j \preccurlyeq \xi_k$\\
and for $i\neq j$, $k\neq \ell$, that $\langle i,j \rangle <^\mathbf{S} \langle k,\ell \rangle$ if and only if $\xi_i \curlywedge \xi_j \prec \xi_k\curlywedge \xi_\ell$,
\item the mappings $w: \{(i,j): i,j\in \mathbb N, i\neq j\} \to \{\curvearrowright, \curvearrowleft\}$ by putting \\ $w(i,j) := W(\xi_i ,\vartheta_i, \xi_j, \vartheta_j)$.
\end{itemize}

Proposition~\ref{P:left_right} then extends $(\mathbb{N}, \, \equiv^{\mathbf{S}}, \, <^{\mathbf{S}}, \,w)$ into  a  random didendritic system $\mathcal D^{\mathbf{S}} =(\mathbb N, \, \equiv^{\mathbf{S}}, \, \langle \cdot, \cdot \rangle^{\mathbf{S}}, \, {<_L}^{\mathbf{S}}, \, {<_R}^{\mathbf{S}}, \, <^{\mathbf{S}})$. The exchangeability of $\mathcal D^\mathbf{S}$ is clear because the random variables $(\xi_i, \vartheta_i)$, $i \in \mathbb N$, are independent and identically distributed. 
The ergodicity of $\mathcal D$ holds because
 the restrictions of the random didendritic system to disjoint finite subsets of $\mathbb{N}$ are independent 
(see Remark 5.18 of \cite{MR3601650}).  
 
 We can now formulate the assertion of Theorem 8.2 of \cite{MR3601650} in a still more explicit and ``constructive'' manner.
 \begin{theorem}\label{T:main}
Let $\mathbf{S}$ be a complete separable  \mbox{$\mathbb{R}$-tree} with root $\rho$,  $\mu$ be a probability measure on $\mathbf{S}$, and  $W$ be a Borel measurable function from $\mathbf{S} \times [0,1] \times \mathbf{S} \times [0,1]$ to \mbox{$\{\curvearrowright, \curvearrowleft\}$}.  Suppose that the properties (T) and (LR) hold. Then $\mathcal D^{\mathbf{S}} =(\mathbb N, \, \mbox{$\equiv^{\mathbf{S}},$} \, \langle \cdot, \cdot \rangle^{\mathbf{S}}, \, {<_L}^{\mathbf{S}}, \, {<_R}^{\mathbf{S}}, \, <^{\mathbf{S}})$ is an ergodic exchangeable random didendritic system. Conversely, for any   ergodic exchangeable random didendritic system $\mathcal D$ there exists a $4$-tuple $(\mathbf{S}, \rho, \mu, W)$ with the abovementioned properties such that $\mathcal D^{\mathbf{S}}$ has the same distribution as $\mathcal D$.
 \end{theorem}

 In short:  The construction described at the beginning of the section builds 
an ergodic exchangeable random didendritic system
(and hence an
extremal infinite PATRICIA bridge) from a rooted $\mathbb R$-tree endowed with a sampling measure $\mu$ and a ``left-right prescription'' $W$ that obey the consistency properties (T) and (LR).  
Conversely, 
any ergodic exchangeable random didendritic system
(and hence any extremal infinite PATRICIA bridge)
arises from such a construction.

The first part of Theorem~\ref{T:main} and its proof has already been explained. 
We now briefly review the proof of the second part.
The arguments in Section~6 of \cite{MR3601650} construct from a given ergodic exchangeable random didendritic system $\mathcal{D} = (\mathbb{N}, \, \equiv, \, \langle \cdot, \cdot \rangle, \, <_L, \, <_R, \, <)$  a~complete separable (ultrametric) $\mathbb{R}$-tree $\mathbf T$ with a distinguished point $\rho \in \mathbf T$, and with an injective mapping from the set of equivalence classes $\{\langle i,j\rangle: i,j \in \mathbb N\}$  into~$\mathbf T$ such that the partial order on~$\mathbf T$ defined by the root $\rho$ extends the partial order $<$ on~$\{\langle i,j\rangle: i,j \in \mathbb N\}$.  The  {\em core} of $\mathbf T$, denoted by $\Gamma(\mathbf T)$, is the closure of the set of points of attachment  the leaves of $\mathbf T$. Here, $\Pi(x)$, the point of attachment of a leaf $x$, equals $x$ if the  leaf is not isolated, whereas for an isolated leaf it is such that the line segment $[x, \Pi(x))$ is the maximal one among all line segments $[x,y)$ that arise as intersections of $\mathbf T$ with open balls centered at $x$. 
Now Proposition~7.4 of \cite{MR3601650}  constructs on the complete separable $\mathbb{R}$-tree $\mathbf{S}:= \Gamma(\mathbf T)$ with root $\rho$ a diffuse probability measure $\mu$ having property (T), and such that $(\mathbb{N}, \, \equiv, \langle \cdot, \cdot \rangle \, \, <)$ has the same distribution as $(\mathbb{N}, \, \equiv^{\mathbf{S}}, \, \langle \cdot, \cdot \rangle^{\mathbf{S}}, \, <^{\mathbf{S}})$.

In Section~8 of \cite{MR3601650}, using the Aldous-Hoover-Kallenberg theory on exchangeable random arrays, a Borel measurable function $W: \mathbf{S} \times [0,1] \times \mathbf{S} \times [0,1] \to \{\curvearrowright, \curvearrowleft\}$ is constructed which has the above property (LR), and which is such that the resulting ergodic exchangeable random didendritic system 
$\mathcal D$ has the same distribution as 
the resulting ergodic exchangeable random didendritic system
$\mathcal D^{\mathbf{S}}$.

\begin{example}\label{E:zigzagfinish}
We continue Examples  \ref{E:zig-zag_path} and \ref{E:zigzagDDS}. Here, the ultrametric $\mathbb{R}$-tree $\mathbf T$ may be taken to be the interval $[0, \frac 12]$, along with disjoint segments of length $\frac 12(1-U_i)$ attached to the points $\xi_i = \frac 12 U_i$, and with the root $\rho = 0$. This is the same description of $\mathbf T$ as in Example 6.7 of \cite{MR3601650}, except that the roles of $0$ and $\frac12$ have been interchanged, in order to tie in with the way the order $<$  is constructed in Example~\ref{E:zigzagDDS}. The core $\mathbf{S} = \Gamma(\mathbf T)$ is then the interval $[0, \frac 12]$, and the sampling measure $\mu$ is the uniform distribution on this interval. The prescription of ``left versus right'' is then determined by the function
\[
W(x,s,y,t) 
= 
\begin{cases}
\curvearrowright,& \quad \text{if $x < y$ and $s < \frac{1}{2}$}, \\
\curvearrowleft,& \quad \text{if $x < y$ and $s > \frac{1}{2}$}, \\
\curvearrowright,& \quad \text{if $y < x$ and $t < \frac{1}{2}$}, \\
\curvearrowleft,& \quad \text{if $y < x$ and $t > \frac{1}{2}$}, \\
\curvearrowleft,& \quad \text{otherwise}.
\end{cases}
\]
\end{example}
\begin{example}
We know from Remark~\ref{R:ePb} that $(\prescript{\nu}{}{\bar R}_n)_{n \in \mathbb{N}}$ is an extremal infinite PATRICIA bridge for each each diffuse probability measure $\nu$ on $\{0,1\}^\infty$. Here the ultrametric $\mathbb{R}$-tree $\mathbf T$ may be taken as follows (cf. Examples 6.8 \& 8.4 of \cite{MR3601650}):
Take the complete binary tree $\{0,1\}^*$, join two elements of the form
$v_1 \ldots v_k$ and $v_1 \ldots v_k v_{k+1}$ with a segment of length
$\frac{1}{2^{k+2}}$, and let  $\mathbf T$ be the completion of this $\mathbb{R}$-tree. The root  $\rho \in \mathbf T$  is the point corresponding to the root $\emptyset \in \{0,1\}^*$. The core $\mathbf{S}=\Gamma(\mathbf T)$ is just $\mathbf T$ itself.
There is a bijective correspondence between $\{0,1\}^\infty$ and the points ``added'' in passing to the completion. The sampling measure $\mu$ on $\mathbf{S}$ is identified via this correspondence with the probability measure $\nu$ on $\{0,1\}^\infty$. Let $x$ and $y$ be two points in the support of
$\mu$ that correspond to the points $u$ and $v$ in $\{0,1\}^\infty$.  The most recent common ancestor of $x$ and $y$ in $\mathbf T$ (that is, the point $z$ such that $[\rho,z] = [\rho,x]\cap[\rho,y]$) is the point in $\mathbf T$ corresponding to the most recent common ancestor of $u$ and $v$ in $\{0,1\}^*$. The ``left versus right'' rule  is then given by 
\[
W(x,s,y,t) 
= 
\begin{cases}
\curvearrowright,& \quad \text{if $u\wedge v  <_L u$ and $u\wedge v  <_R v$}, \\
\curvearrowleft,& \quad \text{if $u\wedge v  <_L v$ and $u\wedge v  <_R u$}.
\end{cases}
\]

\end{example}
\bigskip
\noindent
{\bf Acknowledgement.} We thank the referee of \cite{MR3734107} for suggesting 
that we extend our analysis of the radix sort chain to the PATRICIA chain and
the referee of this paper for helpful and stimulating comments. We are grateful
to Ralph Neininger for sharing his example in Section~\ref{S:heights}.
We acknowledge the hospitality of  the Institute for Mathematical Sciences, 
National University of Singapore, where part of this paper was written.

\def\cprime{$'$}

\end{document}